\documentclass[a4paper,10pt]{article}

\usepackage{amssymb}
\usepackage{amsmath}
\usepackage{amsthm}
\usepackage{multicol}
\usepackage{url}

\numberwithin{equation}{section}

\theoremstyle{plain}
\newtheorem{theorem}{Theorem}[section]
\newtheorem{lemma}{Lemma}[section]

\theoremstyle{definition}
\newtheorem{definition}{Definition}[section]
\newtheorem{algorithm}{\normalfont\textsc{Algorithm}}
\newtheorem{assumption}{Assumption A.\!\!}

\newtheorem{scheme}{\noindent\normalfont\textsc{Scheme A.\!\!}}
\theoremstyle{remark}
\newtheorem{remark}{Remark}[section]

\newcommand{\h}{\mathcal{H}}

\newcommand{\norm}[1]{\|#1\|}

\newcommand{\eofproof}{\hfill$\square$}

\renewcommand{\int}{\mathrm{int}}

\title{Combining Convex-Concave Decompositions and Linearization Approaches for solving BMIs, with application to Static Output Feedback}
\author{Tran Dinh Quoc\footnote{Department of Electrical Engineering (ESAT-SCD) and Optimization in Engineering Center (OPTEC), K.U. Leuven, Kasteelpark Arenberg 10, B-3001
Leuven, Belgium. Email: \{quoc.trandinh,moritz.diehl\}@ esat.kuleuven.be} {~\LARGE{$\cdot$}} Suat Gumussoy\footnote{Department of Computer Science, Katholieke
Universiteit Leuven, Leuven, Belgium. Email: \{suat.gumussoy, wim.michiels\}@cs.kuleuven.be} {~\LARGE$\cdot$} Wim~Michiels\footnotemark[2] {~\LARGE$\cdot$}
Moritz Diehl\footnotemark[1]}
\date{Technical Report, July 2011}

\begin{document}

\maketitle

\begin{abstract}
A novel optimization method is proposed to minimize a convex function subject to bilinear matrix inequality (BMI) constraints. The key idea is to decompose the
bilinear mapping as a difference between
two positive semidefinite convex mappings. At each iteration of the algorithm the concave part is linearized, leading to a convex subproblem.Applications to
various output feedback controller
synthesis problems are presented.  In these applications the subproblem in each iteration step can be turned into  a convex optimization problem with linear
matrix inequality (LMI) constraints.   The
performance of the algorithm has been benchmarked on the data from COMPl$_\textrm{e}$ib library.

\end{abstract}

\section{Introduction}
Optimization involving matrix constraints have broad interest and applications in static state/output feedback controller design, robust stability of
systems, topology optimization (see, e.g. \cite{BenTal2004,Boyd1994,Leibfritz2004,Kocvara2005}).
Many problems in these fields can be reformulated as an optimization problem of linear matrix inequality (LMI) constraints \cite{Boyd1994,Leibfritz2004} which
can be solved efficiently and reliably by means of interior point methods for semidefinite programming (SDP) \cite{BenTal2004,Nesterov1994} and efficient
open-source software tools such as Sedumi \cite{Sturm1999}, SDPT3
\cite{Tutunku2001}.
However, solving optimization problems involving nonlinear matrix inequality constraints is still a big challenge in practice. The methods and algorithms for
nonlinear matrix constrained optimization problems are still limited \cite{Correa2004,Freund2007,Kocvara2005}.

In control theory, many problems related to the design of a reduced-order controller can be conveniently  reformulated as a feasibility problem or an
optimization problem with bilinear matrix inequality (BMI) constraints by means of, for instance, Lyapunov's theory.
The BMI constraints make the problems much more difficult than the LMI ones due to their nonconvexity and possible nonsmoothness. It has been shown in
\cite{Blondel1997} that the optimization problems involving BMI are NP-hard.
Several approaches to solve optimization problems with BMI constraints have been proposed.
For instance, Goh \textit{et al} \cite{Goh1995} considered problems in robust control by means of BMI optimization using global optimization methods. 
Hol \textit{et al} in \cite{Hol2005} proposed to used a sum-of-squares approach to fixed order $H$-infinity synthesis. Apkarian and Tuan \cite{Apkarian1998}
proposed local and global methods for solving BMIs also based on techniques of global optimization. These authors further considered this problem by proposing
parametric formulations and difference of two convex functions (DC) programming approaches. A similar approach can be found in \cite{Alamo2008}.
However, finding a global optimum is in general impractical while global optimization methods are usually recommended to a low dimensional problem. Our method
developed in this paper is classified as a local optimization method which aims as finding a local optimum based on solving a sequence of convex semidefinite
programming problems. 
Sequential semidefinite programming method for nonlinear SDP and its application to robust control was considered by Fares \textit{et al} in \cite{Fares2002}.
Thevenet \textit{et al} \cite{Thevenet2006} studied spectral SDP methods for solving problems involving BMI arising in controller design.
Another approach is based on the fact that problems with BMI constraints can be reformulated as problems with LMI constraints with additional rank constraints.
In  \cite{Orsi2006} Orsi \textit{et al}
developed  a Newton-like method for solving problems of this type.

In this paper, we are interested in optimization problems arising in static output feedback controller design for a linear, time-invariant system of the form:
\begin{equation}\label{eq:LTI}
\left\{\begin{array}{cl}
&\dot{x} = Ax + B_1w + Bu,\\
&z = C_1x + D_{11}w + D_{12}u,\\
&y = Cx + D_{21}w,
\end{array}\right.
\end{equation}
where $x\in\mathbb{R}^{n}$ is state vector, $w\in\mathbb{R}^{n_w}$ is the performance input, $u\in\mathbb{R}^{n_u}$ is input vector, $z\in\mathbb{R}^{n_z}$ is
the performance output,
$y\in\mathbb{R}^{n_y}$ is physical output vector, $A\in\mathbb{R}^{n\times n}$ is state matrix,
$B\in\mathbb{R}^{n\times n_u}$ is input matrix and $C\in\mathbb{R}^{n_y\times n}$ is the output matrix.
Using a static feedback controller of the form $u = Fy$ with $F\in\mathbb{R}^{n_u\times n_y}$, we can write the closed-loop system as follows:
\begin{equation}\label{eq:ss_LTI}
\left\{\begin{array}{cl}
\dot{x}_F = A_Fx_F + B_Fw,\\
z = C_Fx_F + D_Fw.
\end{array}\right.
\end{equation}
The stabilization, $\h_2$, $\h_\infty$ optimization and other control problems for this closed-loop system will be considered.

\vskip 0.1cm
\noindent\textbf{Contribution. }
Many control problems can be expressed as optimization problems of BMI constraints and these optimization problems can conveniently be  reformulated as
optimization problems of difference of two
\textit{positive semidefinite convex} (psd-convex) mappings (or convex-concave decomposition) constraints (see Definition \ref{de:psd_convex} below). In this
paper, we propose to use this
reformulation leading to a new local optimization method for solving some classes of optimization problems involving BMI constraints. We provide a practical
algorithm and prove the convergence of the
algorithm under certain standard assumptions.

The algorithm proposed in this paper is very simple to implement by using available SDP software tools. Moreover, it does not require any globalization strategy
such as line-search procedures to
guarantee global convergence to a local minimum. The method still works in practice for nonsmooth optimization problems, where the objective function and the
concave parts are only subdifferentiable, but not necessarily differentiable. 
Note that our method is different from the standard DCA approach in \cite{Pham1998} since we work directly with positive semidefinite matrix inequality
constraints instead of transforming into DC representations as in \cite{Apkarian1998}.

We show that our method is applicable to many control problems in static state/output feedback controller design. The numerical results are benchmarked using
the data from COMPl$_\textrm{e}$ib
library. Note, however, that this method is also applicable to other nonconvex optimization problems with matrix inequality constraints
which can be written as a convex-concave decomposition.

\vskip 0.1cm
\noindent\textbf{Outline of the paper. }
The remainder of the paper is organized as follows. Section~\ref{sec:premi} provides some preliminary results which will be used in what follows.
Section~\ref{sec:MDCP} presents the formulation of
optimization problems involving convex-concave matrix inequality constraints and a fundamental assumption, Assumption A\ref{as:A1}. The algorithm and its
convergence results are presented in
Section~\ref{sec:SSDP}. Applications to control problems on static feedback controller design and numerical benchmarking are given in
Section~\ref{sec:applications}. The last section contains some
concluding remarks.

\section{Preliminaries}\label{sec:premi}
Let $\mathcal{S}^p$ be the set of symmetric matrices of size $p\times p$, $\mathcal{S}^p_{+}$, and resp., $\mathcal{S}^p_{++}$ be the set of symmetric positive
semidefinite, resp., symmetric positive
definite matrices. For
given matrices $X$ and $Y$ in $\mathcal{S}^p$, the relation $X\succeq Y$ (resp., $X\preceq Y$) means that $X-Y\in\mathcal{S}^p_{+}$ (resp.,
$Y-X\in\mathcal{S}^p_{+}$) and $X\succ Y$ (resp., $X\prec
Y$)
is $X-Y\in\mathcal{S}^p_{++}$ (resp., $Y-X\in\mathcal{S}^p_{++}$). The quantity $X\circ Y := \textrm{trace}(X^TY)$ is an inner product of two matrices $X$ and
$Y$ defined on $\mathcal{S}^p$, where
$\textrm{trace}(Z)$ is the trace of matrix $Z$.
\begin{definition}\label{de:psd_convex}\cite{Shapiro1997}
A matrix-valued mapping $G : \mathbb{R}^n\to\mathcal{S}^p$ is said to be positive semidefinite convex (\textit{psd-convex}) on a convex subset
$C\subseteq\mathbb{R}^n$ if for all $t\in[0,1]$ and
$x,y\in C$, one has
\begin{equation}\label{eq:psd_convex}
G(tx+(1-t)y) \preceq tG(x) + (1-t)G(y).
\end{equation}
If \eqref{eq:psd_convex} holds true for $\preceq$ instead of $\prec$ and $t\in (0,1)$ then $G$ is said to be \textit{strictly psd-convex} on $C$.
Alternatively, if we replace $\preceq$ in \eqref{eq:psd_convex} by $\succeq$ then $G$ is said to be psd-concave on $C$.
It is obvious that any convex function $f : \mathbb{R}^n\to\mathbb{R}$ is psd-convex ($p=1$).
\end{definition}
A function $f :\mathbb{R}^n\to\mathbb{R}$ is said to be \textit{strongly convex} with the parameter $\rho > 0$ if $f(\cdot) - \frac{1}{2}\rho\norm{\cdot}^2$ is
convex.

The derivative of a matrix-valued mapping $G$ at $x$ is a linear mapping $DG$ from $\mathbb{R}^n$ to $\mathbb{R}^{p\times p}$ which is defined by
\begin{equation*}
DG(x)h := \sum_{i=1}^nh_i\frac{\partial{G}}{\partial x_i}(x), ~\forall h\in \mathbb{R}^n.
\end{equation*}
\vskip -0.1cm
For a given convex set $X\in\mathbb{R}^n$, the matrix-valued mapping $G$ is said to be differentiable on a subset $X$ if its derivative $DG(x)$ exists at every
$x\in X$. The definitions
of the second order derivatives of matrix-valued mappings can be found, e.g., in \cite{Shapiro1997}.
Let $A : \mathbb{R}^n\to \mathcal{S}^p$ be a linear mapping defined as $Ax = \sum_{i=1}^nx_iA_i$, where $A_i \in \mathcal{S}^p$ for $i=1,\dots, n$.
The adjoint operator of $A$, $A^{*}$,  is defined as $A^{*}Z = (A_1\circ Z, A_2\circ Z, \dots, A_n\circ Z)^T$ for any $Z\in\mathcal{S}^p$.

\begin{lemma}\label{le:psd_convex_properties}
\begin{itemize}
\item[a)] A matrix-valued mapping $G$ is psd-convex on $X$ if and only if for any $v\in\mathbb{R}^p$ the function $\varphi(x) := v^TG(x)v$ is convex on $X$.
\item[b)] A mapping $G$ is psd-convex on $X$ if and only if for all $x$ and $y$ in $X$, one has
\begin{equation}\label{eq:psd_convex_pro1}
G(y) - G(x) \succeq DG(x)(y-x).
\end{equation}
\end{itemize}
\end{lemma}
\begin{proof}
The proof of the statement a) can be found in \cite{Shapiro1997}. We prove b). Let $\varphi(x) = v^TG(x)v$ for any $v\in\mathbb{R}^p$.
If $G$ is psd-convex then $\varphi$ is convex. We have $\varphi(y) - \varphi(x) \geq \nabla\varphi(x)^T(y-x)$. Now, $\nabla\varphi(x)^T(y-x) =
\sum_{i=1}^n(y_i-x_i)v^T\frac{\partial{G}}{\partial{x}_i}(x)v = v^T[DG(x)(y-x)]v$.  
Hence, $v^T\left[G(y)-G(x)-DG(x)\right.$ $\left.(y-x)\right]v \geq 0$ for all $v$. We conclude that \eqref{eq:psd_convex_pro1} holds.
Conversely, if \eqref{eq:psd_convex_pro1} holds then, for any $v$, we have $v^T\left[G(y) - G(x) - DG(x)(y-x)\right]v \geq 0$, which is equivalent to
$\varphi(y) - \varphi(x) \geq
\nabla\varphi(x)^T(y-x)$. Thus $\varphi$ is convex. By virtue of a), the mapping $G$ is psd-convex.
\end{proof}

For simplicity of discussion, throughout this paper, we assume that all the functions and matrix-valued mappings are \textit{twice differentiable} on their
domain \cite{Shapiro1997,Thevenet2006}.
However, this assumption can be reduced to the \textit{subdifferentiability} of the objective function and the concave parts of the matrix-valued mappings.
\begin{definition}\label{de:DC_decomp}
A matrix-valued mapping $F :\mathbb{R}^n\to\mathcal{S}^p$ is said to be a \textit{psd-convex-concave} mapping if $F$ can be represented as a difference of two
psd-convex mappings, i.e.
$F(x) = G(x) - H(x)$, where $G$ and $H$ are psd-convex.
The pair $(H,G)$ is called a psd-DC (or psd-convex-concave) decomposition of $F$.
\end{definition}
Note that each given psd-convex-concave mapping possesses many psd-convex-concave decompositions.

\section{Optimization of convex-concave matrix inequality constraints}\label{sec:MDCP}

\subsection{Psd-convex-concave decomposition of BMIs}
Instead of using the vector $x$ as a decision variable, we use from now on the matrix $X$ as a matrix variable in $\mathbb{R}^{m\times n}$.
Note that any matrix $X$ can be considered as an $m\times n$-column vector by vectorizing with respect to its columns, i.e. $x = \textrm{vec}(X) := (X_{11},
X_{21}, \dots, X_{mn})^T$. The inverse
mapping of $\textrm{vec}$ is called $\textrm{mat}$.
Since $\textrm{vec}$ and $\textrm{mat}$ are linear operators, the psd-convexity is still preserved under these operators.

A mapping $F : \mathbb{R}^{p\times q}\times\mathcal{S}^{p} \to \mathcal{S}^p$ given by  $F(X,Y) := XQ^{-1}X^T - Y$, where $Q\in\mathcal{S}^q$ is symmetric
positive definite, is called a
\textit{Schur psd-convex}\footnote{Due to Schur's complement form} mapping.

Consider a bilinear matrix form
\begin{equation}\label{eq:bilinearform}
F(X, Y) := X^TY + Y^TX.
\end{equation}
By using the Kronecker product, we can write $F$ as $\textrm{vec}(F(X,Y)) = (I_n\otimes X^T)\textrm{vec}(Y) + (I_y\otimes Y^T)\textrm{vec}(X) =
(\sum_{i,j}x_iy_j)$, where $I_n$, $I_y$ are appropriate
identity matrices, $\otimes$ denotes the Kronecker product. Hence, the vectorization of $F(X,Y)$ is indeed a bilinear form of two vectors $x := \textrm{vec}(X)$
and $y:=\textrm{vec}(Y)$.

The following lemma shows that the bilinear matrix form \eqref{eq:bilinearform} can be decomposed as a difference of two psd-convex mappings.
\begin{lemma}\label{le:key_lemma_a}
\begin{itemize}
\item[a)] The mapping $f(X) := X^TX$, $g(X) := XX^T$) are psd-convex on $\mathbb{R}^{m\times n}$. The mapping $f(X) := X^{-1}$ is psd-convex on
$\mathcal{S}^{p}_{++}$.
\item[b)] The bilinear matrix form $X^TY + Y^TX$ can be represented as a psd-convex-concave mapping of at least three forms:
\begin{align}\label{eq:DC_form}
X^TY + Y^TX & =  (X+Y)^T(X+Y) - (X^TX + Y^TY)\nonumber\\
& = X^TX + Y^TY - (X-Y)^T(X-Y)\\ 
& = \frac{1}{2}[(X+Y)^T(X+Y) - (X-Y)^T(X-Y)].\nonumber
\end{align}
\end{itemize}
\end{lemma}
The statement b) provides at least three different explicit psd-convex-concave decompositions of the bilinear form $X^TY+Y^TX$. 
Intuitively, we can see that the first decomposition has a ``strong curvature'' on the second term, while the second and the third decompositions have ``less
curvature'' on the second term due in case
of a compensation between $X$ and $Y$.

The following result will be used  to transform  Shur psd-convex constraints to LMI constraints.
\begin{lemma}\label{le:key_lemma}
\begin{itemize}
\item[a)] Suppose that $A\in\mathcal{S}^n$. Then the matrix inequality $BB^T - A \prec (\preceq)~ 0$ is equivalent to
\begin{equation}\label{eq:Schur_comp}
\begin{bmatrix}A & B\\ B^T & I\end{bmatrix} \succ (\succeq)~ 0.
\end{equation}
\item[b)] Suppose that $A\in\mathcal{S}^n$, $D \succ 0$, then we have:
\begin{equation}
\begin{bmatrix}
A - BB^T & C\\ C^T & D
\end{bmatrix} \succ (\succeq) ~ 0 ~~
\Longleftrightarrow~~
\begin{bmatrix}
A & B & C\\ B^T & I & O\\ C^T & O & D
\end{bmatrix} \succ (\succeq)~0.
\end{equation}
\end{itemize}
\end{lemma}
The proof of this lemma immediately follows by applying Schur's complement and Lemma \ref{le:psd_convex_properties} \cite{Boyd2004}. We omit the proof here.

\subsection{Optimization involving convex-concave matrix inequality constraints}
Let us consider the following optimization problem:
\begin{equation}\label{eq:MDC_program}
\left\{\begin{array}{cl}
\displaystyle\min_{x}&f(x)\\
\textrm{s.t.}& G_i(x) - H_i(x) \preceq 0, ~i=1,\dots, l,\\
&x\in\Omega,
\end{array}\right.
\end{equation}
where $f :\mathbb{R}^n\to\mathbb{R}$ is convex, $\Omega\subseteq\mathbb{R}^n$ is a nonempty, closed convex set, and $G_i$ and $H_i$ ($i=1,\dots,l$) are
psd-convex. Problem \eqref{eq:MDC_program} is
referred to as a convex optimization with psd-convex-concave matrix inequality constraints.

Let $\Omega$ be a polyhedral in $\mathbb{R}^n$. Then, if $f$ is nonlinear or one of the mappings $G_i$ or $H_i$ ($i=1,\dots,l)$ is nonlinear then
\eqref{eq:MDC_program} is a nonlinear
semidefinite program. If $H_i$ ($i=1,\dots, l$) are linear then
\eqref{eq:MDC_program} is a convex nonlinear SDP problem. Otherwise, it is a nonconvex nonlinear SDP problem.

Let us define $L(x,\Lambda) := f(x) + \sum_{i=1}^l\Lambda_i\circ [G_i(x)-H_i(x)]$ as the Lagrange function of \eqref{eq:MDC_program}, where
$\Lambda_i\in\mathcal{S}^p$ ($i=1,\dots, l$) considered as
Lagrange multipliers. The generalized KKT condition of \eqref{eq:MDC_program} is presented as:
\begin{equation}\label{eq:DC_KKT}
\left\{\begin{array}{cl}
&0 \in \nabla f(x) + \sum_{i=1}^lD[G_i(x)-H_i(x)]^{*}\Lambda_i + N_{\Omega}(x),\\
&G_i(x)-H_i(x)\preceq 0,~ \Lambda_i\succeq 0,\\
&[G_i(x)-H_i(x)]\circ\Lambda_i=0, ~i=1,\dots, l.
\end{array}\right.
\end{equation}
Here, $N_{\Omega}(x)$ is the normal cone of $\Omega$ at $x$ defined as
\begin{equation*}
N_{\Omega}(x) := \left\{\begin{array}{cl}
\left\{w\in\mathbb{R}^n ~|~ w^T(y-x)\geq 0, ~\forall y\in \Omega\right\}, &~\textrm{if}~ x\in\Omega,\\
\emptyset, &~\textrm{otherwise}.
\end{array}\right.
\end{equation*}
A pair $(x^{*}, \Lambda^{*})$ satisfying \eqref{eq:DC_KKT} is called a KKT point, $x^{*}$ is called a stationary point and $\Lambda^{*}$
is the corresponding multiplier of \eqref{eq:MDC_program}. The generalized optimality condition for nonlinear semidefinite programming can be found in the
literature (e.g.,
\cite{Shapiro1997,Sun2006}).

Let us denote by
\begin{equation}\label{eq:feasible_set}
\mathcal{D} := \left\{x\in\Omega ~|~ G_i(x)-H_i(x)\preceq 0, ~l=1,\dots,l\right\},
\end{equation}
the feasible set of \eqref{eq:MDC_program} and $\textrm{ri}(\mathcal{D})$ is the  relative interior of $\mathcal{D}$ which is defined by
\begin{equation*}
\textrm{ri}(\mathcal{D}) := \left\{
x\in\textrm{ri}(\Omega) ~|~ G_i(x) - H_i(x) \prec 0, ~i=1,\dots, m\right\},
\end{equation*}
where $\textrm{ri}(\Omega)$ is the set of classical relative interiors of $\Omega$ \cite{Boyd2004,Horst1996}.
The following condition is a fundamental assumption in this paper.

\begin{assumption}\label{as:A1}
$\textrm{ri}(\mathcal{D})$ is nonempty.
\end{assumption}
Note that this assumption is crucial for our method, because, as we shall see, it requires a strictly feasible starting point $x^0\in\textrm{ri}(\mathcal{D})$.
Finding such a point is in principle not
an easy task.
However, in many problems, this assumption is always satisfied.
In Section~\ref{sec:applications} we will propose techniques to determine a starting point for the control problems under consideration.

\section{The algorithm and its convergence}\label{sec:SSDP}
In this section, a local optimization method for finding a stationary point of problem \eqref{eq:MDC_program} is proposed.
Motivated from the DC programming algorithm developed in \cite{Pham1998} and the convex-concave procedure in \cite{Sriperumbudur2009} for scalar functions, we
develop an iterative procedure for
finding a stationary point of \eqref{eq:MDC_program}.
The main idea is to linearize the nonconvex part of the psd-convex-concave matrix inequality constraints and then transform the linearized subproblem into a
quadratic semidefinite programming
problem. The subproblem can be either directly solved by means of interior point methods or transformed into a quadratic problem with LMI constraints. In the
latter case,  the resulting
problem can be solved by available software tools such as Sedumi \cite{Sturm1999} and SDPT3 \cite{Tutunku2001}.

\subsection{The algorithm}
Suppose that $x^k\in\Omega$ is a given point, the linearized problem of \eqref{eq:MDC_program} around $x^k$ is written as
\begin{equation}\label{eq:psd_convex_subprob}
\left\{ \begin{array}{cl}
\displaystyle\min_{x}  & \left\{ f_k(x) := f(x) + \frac{\rho_k}{2}\|Q_k(x-x_k)\|^2_2\right\}\\
\textrm{s.t.}  &\!\!\! G_i(x) \!-\! H_i(x^k) \!-\! DH_i(x^k)(x \!-\! x^k) \! \preceq \!0, \!~i=1,\dots, l,\\
&\!\!\! x\in\Omega.
\end{array}\right.
\end{equation}
Here, we add a regularization term into the objective function of the original problem, where $Q_k$ is a given matrix that projects $x-x_k$ in a certain
subspace of $\mathbb{R}^n$ and $\rho_k \geq 0$
is a regularization parameter.
Since $G_i$ ($i=1,\dots, l$) are psd-convex and the objective function is convex, problem \eqref{eq:psd_convex_subprob} is convex.
The \textit{linearized convex-concave SDP algorithm} for solving \eqref{eq:MDC_program} is described as follows.

\begin{algorithm}\label{alg:A1}{~}\\
\noindent\textbf{Initialization:} Choose a positive number $\rho_0$ and a matrix $Q_0\in\mathcal{S}^n_{+}$. Find an initial point
$x^0\in\textrm{ri}(\mathcal{D})$. Set $k := 0$.\\
\noindent\textbf{Iteration $k$:} For $k=0, 1,\dots$. Perform the following steps:
\begin{itemize}
\item[]\textit{Step 1:} Solve the convex semidefinite program \eqref{eq:psd_convex_subprob} to obtain a solution $x^{k+1}$.
\item[]\textit{Step 2:} If $\norm{x^{k+1} - x^{k}}\leq \varepsilon$ for a given tolerance $\varepsilon>0$ then terminate. Otherwise, update $\rho_k$ and $Q_k$
(if necessary), set $k:=k+1$ and go back
to \textit{Step 1}.
\end{itemize}
\end{algorithm}
The following main property of the method makes an implementation very easy.
If the initial point $x^0$ belongs to the relative interior of the feasible set $\mathcal{D}$, i.e. $x^0\in\textrm{ri}(\mathcal{D})$, then Algorithm
\ref{alg:A1} generates a
sequence ${x^k}$ which still belongs to $\mathcal{D}$.
Consequently, no line-search procedure is needed to ensure the global convergence.

This property follows from the fact that the linearization of the concave part $-H_i$ is its an upper approximation of this mapping (in the sense of positive
semidefinite cone), i.e.
\begin{equation*}
-H_i(x) \preceq -H_i(x^k) - DH_i(x^k)(x-x^k), \forall x\in\Omega,
\end{equation*}
which is equivalent to
\begin{equation*}
G_i(x) - H_i(x) \preceq G_i(x) - H_i(x^k) - DH_i(x^k)(x-x^k), \forall x\in\Omega.
\end{equation*}
Hence, if the subproblem \eqref{eq:psd_convex_subprob} has a solution $x^{k+1}$ then it is feasible to \eqref{eq:MDC_program}.
Geometrically, Algorithm \ref{alg:A1} can be seen as an inner approximate method.

The main tasks of an implementation of Algorithm \ref{alg:A1} consist of:
\begin{itemize}
\item determining an initial point $x^0\in\textrm{ri}(\mathcal{D})$, and 
\item solving the convex semidefinite program \eqref{eq:psd_convex_subprob} repeatedly.
\end{itemize}
As mentioned before, since $\mathcal{D}$ is nonconvex, finding an initial point $x^0$ in $\textrm{ri}(\mathcal{D})$ is, in principle, not an easy task. However,
in some practical problems, this can
be done by exploiting the special structure of the problem (see the examples in Section~\ref{sec:applications}).

To solve the convex subproblem \eqref{eq:psd_convex_subprob}, we can either implement an interior point method and exploit the structure of the problem or
transform it into a standard SDP problem and
then make use of available software tools for SDP.
The regularization parameter $\rho_k$ and the projection matrix $Q_k$ can be fixed at appropriate choices for all iterations, or adaptively updated.
\begin{lemma}\label{le:KKT_point}
If $x^{k}$ is a solution of \eqref{eq:psd_convex_subprob} linearized at $x^k$ then it is a stationary point of \eqref{eq:MDC_program}.
\end{lemma}
\begin{proof}
Suppose that $\Lambda^{k+1}$ is a multiplier associated with $x^k$, substituting $x^k$ into the generalized KKT condition \eqref{eq:KT_cond} of
\eqref{eq:psd_convex_subprob} we obtain
\eqref{eq:DC_KKT}.
Thus $x^k$ is a stationary point of \eqref{eq:MDC_program}.
\end{proof}

\subsection{Convergence analysis}
\vskip -0.3cm
In this subsection, we restrict our discussion to the following special case.
\begin{assumption}\label{as:A2}
The mappings $G_i$ ($i=1,\dots, l$) are Schur psd-convex and $\Omega$ is formed by a finite number of LMIs. In addition, $f$ is convex quadratic on
$\mathbb{R}^n$
with a convexity parameter $\rho_f\geq 0$.
\end{assumption}
This assumption is only technical for our implementation. If the mapping $G_i$ is Schur psd-convex then the linearized constraints of problem
\eqref{eq:psd_convex_subprob} can directly be transformed
into LMI constraints (see Lemma \ref{le:key_lemma}). In practice, $G_i$ ($i=1,\dots,l$) can be a general psd-convex mappings and $f$ can be a general convex
function.

Under Assumption A\ref{as:A2}, the convex subproblem \eqref{eq:psd_convex_subprob} can be transformed equivalently into a quadratic semidefinite program of the
form:
\begin{equation}\label{eq:LSDP}
\left\{\begin{array}{cl}
\displaystyle\min_{z\in\mathbb{R}^{n_z}} & \frac{1}{2}z^TBz + h^Tz\\
\textrm{s.t.} &A(z) + C \preceq 0,
\end{array}\right.
\end{equation}
where $A$ is a linear mapping from $\mathbb{R}^{n_z}$ to $\mathcal{S}^{p_z}$, $C\in\mathcal{S}^{p_z}$ and $B$ is a symmetric matrix, by means of Lemma
\ref{le:key_lemma}.

A vector $\hat{z}$ is said to satisfy the \textit{Slater condition} of \eqref{eq:LSDP} if $A(\hat{z}) + C \prec 0$. Suppose that the triple
$(\bar{z},\bar{V},\bar{S})$ satisfies the KKT condition
of \eqref{eq:LSDP} (see \cite{Freund2007}), where $\bar{z}$ is a primal stationary point, $\bar{V}$ is a Lagrange multiplier and $\bar{S}$ is a slack variable
associated with $\bar{z}$ and $\bar{V}$.
Then, problem \eqref{eq:LSDP} is said to satisfy the \textit{strict complementarity} condition at $(\bar{z},\bar{V},\bar{S})$ if $\bar{V} + \bar{S} \succ 0$.

Let $\bar{z}$ be a stationary point of \eqref{eq:LSDP}. We say that $0\neq p\in\mathbb{R}^{n_z}$ is a \textit{feasible direction} to \eqref{eq:LSDP} if $\bar{z}
+ \varepsilon p$ is a feasible point of
\eqref{eq:LSDP} for all $\varepsilon>0$ sufficiently small.
As in \cite{Freund2007}, we assume that the \textit{second order sufficient condition} holds for \eqref{eq:LSDP} at $\bar{z}$ with modulus $\mu > 0$ if for all
feasible directions $p$ at $\bar{z}$
with $p^T(h+B\bar{z}) = 0$, one has $p^TBp \geq \mu\norm{p}^2$.
We say that the convex problem \eqref{eq:LSDP} is solvable and satisfies the \textit{strong second order sufficient condition} if there exists a KKT point
$(\bar{z},\bar{V},\bar{S})$ of the KKT
system of \eqref{eq:LSDP} satisfies the second order sufficient condition and the strict complementary condition.

\begin{assumption}\label{as:A3}
The convex subproblem \eqref{eq:psd_convex_subprob} is \textit{solvable} and satisfies the \textit{strong second order sufficient condition}.
\end{assumption}
Assumption A\ref{as:A3} is standard in optimization and is usually used to investigate the convergence of the algorithms
\cite{Fares2002,Freund2007,Shapiro1997}.

The following lemma shows that $\Delta x^k := x^{k+1}-x^k$ is a descent direction of problem \eqref{eq:MDC_program} whose proof is given in the Appendix. %
\ref{app:A1}.
\begin{lemma}\label{le:descent_dir}
Suppose that $\{(x^k,\Lambda^k)\}_{k\geq 0}$ is a sequence generated by Algorithm \ref{alg:A1}. Then:
\begin{itemize}
\item[]$\mathrm{a)}$ The following inequality holds for $k\geq 0$:
\begin{eqnarray}\label{eq:descent_dir}
f(x^{k + 1}) - f(x^k) \leq -\frac{\rho_f}{2}\norm{x^{k + 1} - x^k}_2^2 - \rho_k\norm{Q_k(x^{k+1} - x^k)}^2_2,
\end{eqnarray}
where $\rho_f$ is the convexity parameter of $f$.
\item[]$\mathrm{b)}$ If there exists at least one constraint $i_0$, $i_0\in\{1,2,\dots,l\}$, to be strictly feasible at $x^k$, i.e. $G_{i_0}(x^k) - H_{i_0}(x^k)
\prec 0$, then $f(x^{k+1}) <
f(x^k)$ provided that $\Lambda^{k+1}_{i_0} \succ 0$.
\item[]$\mathrm{c)}$ If $\rho_k > 0$ and $Q_k$ is full-row-rank then $\Delta x^k$ is a sufficiently descent direction of \eqref{eq:MDC_program}.
\end{itemize}
\end{lemma}
The following theorem shows the convergence of Algorithm \ref{alg:A1} in a particular case.

\begin{theorem}\label{th:convergence}
Under Assumptions A\ref{as:A1}, A\ref{as:A2} and A\ref{as:A3}, suppose that $f$ is bounded from below on $\mathcal{D}$, where $\mathcal{D}$ is assumed to be
bounded in $\mathbb{R}^n$. Let
$\{(x^k,\Lambda^k)\}$ be a sequence generated
by Algorithm \ref{alg:A1} starting from $x^0\in\textrm{ri}(\mathcal{D})$.
Then if either $f$ is strongly convex or $\rho_k \equiv \rho > 0$ and $Q_k \equiv Q$ is full-row-rank for all $k\geq 0$ then every accumulation point
$(x^{*},\Lambda^{*})$ of $\{(x^k, \Lambda^k)\}$ is
a KKT point of
\eqref{eq:MDC_program}.
Moreover, if the set of the KKT points of \eqref{eq:MDC_program} is finite then the whole sequence $\{ (x^k, \Lambda^k) \}$ converges to a KKT point of
\eqref{eq:MDC_program}.
\end{theorem}
\begin{proof}
Let $M(x^0) := \left\{x^k\right\}$ be a sequence of the sample points generated by Algorithm \ref{alg:A1} starting from $x^0$.
For a given $x\in\Omega$, let us define the following mapping:
\vskip -0.3cm
\begin{eqnarray}
A_{\textrm{sol}}(x) := \text{arg}\!\min \Big\{\!\!\!\!\! && f(y) + \frac{\rho}{2}\norm{Q(y - x)}_2^2 ~~|~~ y\in\Omega, \\
&& G_i(y) -  H_i(x) - DH_i(x)(y - x) \preceq 0, ~i=1,\dots,m
\Big\}.\notag
\end{eqnarray}
Then, $A_{\textrm{sol}}$ is a multivalued mapping and it can be considered as the solution mapping of the convex subproblem \eqref{eq:psd_convex_subprob}.
Note that  the sequence $\{x^k\}$ generated by Algorithm \ref{alg:A1} satisfies $x^{k+1} \in A_{\textrm{sol}}(x^k)$.\\
We first prove that $A_{\textrm{sol}}$ is a closed mapping.
Indeed, since the convex subproblem \eqref{eq:psd_convex_subprob} satisfies Slater's condition and has a solution that satisfies the
strict complementarity and the second order sufficient condition, applying Theorem 1 in \cite{Freund2007} we conclude that the mapping $A_{\textrm{sol}}$ is
differentiable in a neighborhood of the
solution. Consequently, it is closed due to the compactness of $\mathcal{D}$.

On the other hand, since $f$ is either strongly convex or $\rho_k \equiv \rho >0$ for all $k\geq 0$ and $Q_k\equiv Q$ is full-row-rank, it follows from Lemma
\ref{le:descent_dir} that the objective
function $f$ is strictly monotone on $M(x^0)$. Since $M(x^0)\subseteq\mathcal{D}$ and $\mathcal{D}$ is compact, $M(x^0)$ is also compact. Applying Theorem 2 in
\cite{Meyer1976} we conclude that every
limit points of the
sequence $\{x^k\}$ belongs to the set of stationary points $S^{*}$.
Moreover, $S^{*}$ is connected and if $S^{*}$ is finite then the whole sequence $\{x^k\}$ converges to $x^{*}$ in $S^{*}$.
\end{proof}

\begin{remark}\label{re:quadratic}
The condition that $f$ is quadratic in Assumption \ref{as:A2} can be relaxed to $f$ being twice continuously differentiable. However, in this case, we need a
direct proof for Theorem
\ref{th:convergence} instead of applying Theorem 1 in \cite{Freund2007}.
\end{remark}

\section{Applications to robust controller design}\label{sec:applications}
In this section, we apply the method developed in the previous section to the following static state/output feedback controller design problems:
\begin{enumerate}
\item Sparse linear static output feedback controller design;
\item Spectral abscissa and pseudo-spectral abscissa optimization;
\item $\mathcal{H}_2$ optimization;
\item $\mathcal{H}_{\infty}$ optimization;
\item and mixed $\mathcal{H}_2/\mathcal{H}_{\infty}$ synthesis.
\end{enumerate}
 We used the system data from \cite{Hassibi1999,Ostertag2008} and the COMP$\textrm{l}_\textrm{e}$ib library \cite{Leibfritz2003}. All the implementations are
done in Matlab 7.11.0 (R2010b)
running on a PC Desktop Intel(R) Core(TM)2 Quad CPU Q6600 with 2.4GHz and 3Gb RAM.
We use the YALMIP package \cite{Lofberg2004} with the SeDuMi 1.1 solver \cite{Sturm1999} to solve the LMI optimization problems arising in Algorithm
\ref{alg:A1} at the initial phase (Phase 1) and subproblem \eqref{eq:psd_convex_subprob}.
The Matlab codes can be downloaded at \url{http://www.kuleuven.be/optec/software/BMIsolver}.
We also benchmarked our method with various examples and compared our results with HIFOO \cite{Gumussoy09} and PENBMI \cite{Henrion2005} for all control
problems. HIFOO is an open-source Matlab
package for fixed-order controller design.
It computes a fixed-order controller using a hybrid algorithm for nonsmooth, nonconvex optimization based on quasi-Newton updating and gradient sampling.
PENBMI \cite{Henrion2005} 
is a commercial software for solving optimization problems with quadratic objective and BMI constraints. PENBMI is free licensed for academic purposes. 
We initialized the initial controller for HIFOO and the BMI parameters for PENBMI to the initial values of our method. As we shall see, we can reformulate the
spectral abscissa
optimization problem as a rank constrained LMI problem. Therefore, we also compared our results with LMIRank \cite{Orsi2006},
a MATLAB toolbox for solving rank constrained LMI problems, for the spectral abscissa optimization.

Note that all problems addressed here lead to at least one BMI constraint. To apply the method developed in the previous section, we propose a unified scheme to
treat these problems.

\begin{scheme}\label{scheme:A1}{~}
\begin{itemize}
\item[] \textit{Step 1. } Find a convex-concave decomposition of the BMI constraints as $G(x)-H(x) \preceq 0$.
\item[] \textit{Step 2. } Find a starting point $x^0\in\textrm{ri}(\mathcal{D})$.
\item[] \textit{Step 3. } For  a given $x^k$, linearize the concave part to obtain the convex constraint $G(x) - H_k(x) \preceq 0$, where $H_k$ is the
linearization of $H$ at $x^k$.
\item[] \textit{Step 4. } Reformulate the convex constraint as LMI constraint by means of Lemma \ref{le:key_lemma}.
\item[] \textit{Step 5. } Apply Algorithm \ref{alg:A1} with SDP solver to solve the problem.
\end{itemize}
\end{scheme}

\subsection{Sparse linear constant output-feedback design}\label{subsec:sparse_SOF}
Let us consider a BMI optimization problem of sparse linear constant output-feedback design given as:
\begin{equation}\label{eq:BMI_p1}
\begin{array}{cl}
\displaystyle\min_{\alpha, P, F} & -\sigma\alpha+\sum_{i=1}^{n_u}\sum_{j=1}^{n_y}|F_{ij}|\\
\text{s.t} ~& (A+BFC)^TP + P(A+BFC) + 2\alpha P \prec 0, ~~P = P^T,~ P\succ 0.
\end{array}
\end{equation}
Here, matrices $A$, $B$, $C$ are given with appropriate dimensions, $P$ and $F$ are referred as variables and $\sigma > 0$ is a weighting parameter. The
objective function consists of two terms: the
first term $\sigma\alpha$ is to stabilize the system (or to maximize the decay rate) and the second one is to ensure the sparsity of the gain matrix $F$.
This problem is a modification of the first example in \cite{Hassibi1999}.
Let us illustrate \textsc{Scheme A.\ref{scheme:A1}} for solving this problem.

\textit{Step 1:} Let $B_F := A + BFC + \alpha I$, where $I$ is the identity matrix. Then, applying Lemma \ref{le:key_lemma_a} we can write
\begin{align}
&(A \!+\! BFC)^T\!\!P \!+\! P(A\!+\!BFC) + 2\alpha P = B_F^TP + PB_F\nonumber\\
& = B_F^TB_F + P^TP - (B_F-P)^T(B_F-P),\label{eq:BMI_p1_dc1}\\
& = \frac{1}{2}\left[(B_F+P)^T(B_F+P) - (B_F-P)^T(B_F-P)\right].\label{eq:BMI_p1_dc2}
\end{align}
In our implementation, we use the decomposition \eqref{eq:BMI_p1_dc2}.

If we denote by
\begin{eqnarray}\label{eq:G_H_mapping}
G(\alpha,P,F) :=\frac{1}{2}(B_F+P)^T(B_F+P), ~~ \textrm{and} ~~ H(\alpha,P,F) :=\frac{1}{2}(B_F-P)^T(B_F-P),
\end{eqnarray}
then the BMI constraint in \eqref{eq:BMI_p1} can be written equivalently as a psd-convex-concave matrix inequality constraint (of a variable $x$ formed from
$(\alpha, P, F)$ as $x := (\alpha,
\text{vec}(P)^T, \textrm{vec}(F)^T)^T$) as follows:
\begin{eqnarray}\label{eq:DC_const_p1}
G(\alpha,P,F) - H(\alpha,P,F) \prec 0.
\end{eqnarray}
Note that the objective function of \eqref{eq:BMI_p1} is convex but nonsmooth which is not directly suitable for the SSDP approach in \cite{Correa2004}, but,
the nonconvex problem \eqref{eq:BMI_p1}
can be reformulated in the form of \eqref{eq:MDC_program} using slack variables.

\textit{Steps 2-5: } The implementation is carried out as follows:
\begin{itemize}
\item[]\textbf{Phase 1. }(\textit{Determine a starting point $x^0\in\textrm{ri}(\mathcal{D})$}). Set $F^0 := 0$, $\alpha^0 := -\alpha_0(A^T+A)/2$ where
$\alpha_0(\cdot)$ is the maximum real part of
the eigenvalues of the matrix, and compute $P=P^0$ as the solution of the LMI feasibility problem
\begin{equation}\label{eq:LMI_phase1}
(A+BF^0 C)^TP + P(A+BF^0 C) + 2\alpha^0 P \prec 0.
\end{equation}
The above choice for $(\alpha^0,F^0)$  originates from the property that  $P^0 = I$ renders the left hand size of (\ref{eq:LMI_phase1}) negative semi-definite
(but not negative definite).
\item[]\textbf{Phase 2. } Perform Algorithm \ref{alg:A1} with a starting point $x^0$ found at Phase 1.
\end{itemize}
Let us now illustrate \textit{Step 4} of \textsc{Scheme A.\ref{scheme:A1}}. After linearizing the concave part of the convex-concave reformulation of the last
BMI constraint in \eqref{eq:BMI_p1} at
$(F^k, P^k, \alpha^k)$ we obtain the linearization:
\begin{equation}\label{eq:Shur_form2}
(A \!+\! BFC \!+\! \alpha I \!+\! P)^T(A \!+\! BFC \!+\! \alpha I \!+\! P) - H_k(F,P,\alpha) \prec 0,
\end{equation}
where $H_k(F,P,\alpha)$ is a linear mapping of $F$, $P$ and $\alpha$.
Now, applying Lemma \ref{le:key_lemma}, \eqref{eq:Shur_form2} can be transformed to an LMI constraint:
\begin{equation*}
\begin{bmatrix}
H_k(F, P, \alpha) & (A \!+\! BFC \!+\! \alpha I \!+\! P)^T\\
(A \!+\! BFC \!+\! \alpha I \!+\! P) & I
\end{bmatrix}
\succ 0.
\end{equation*}
With the above approach we solved problem \eqref{eq:BMI_p1} for the same system data as in \cite{Hassibi1999}. 
Here, matrices $A$, $B$ and $C$ are given as:
\begin{align*}
A \!\!=\!\! {\scriptsize\left[\!\begin{array}{rrrrr}
   \!\!\!\! 2.45  \!\!\!\!&\!\! -0.90 \!\!\!&\!\!  1.53 \!\!\!&\!\! -1.26 \!\!\!&\!\!  1.76\!\!\\
   \!\!\!\! -0.12 \!\!\!\!&\!\! -0.44 \!\!\!&\!\! -0.01 \!\!\!&\!\!  0.69 \!\!\!&\!\!  0.90\!\!\\
   \!\!\!\! 2.07  \!\!\!\!&\!\! -1.20 \!\!\!&\!\! -1.14 \!\!\!&\!\!  2.04 \!\!\!&\!\! -0.76\!\!\\
   \!\!\!\! -0.59 \!\!\!\!&\!\!  0.07 \!\!\!&\!\!  2.91 \!\!\!&\!\! -4.63 \!\!\!&\!\! -1.15\!\!\\
   \!\!\!\! -0.74 \!\!\!\!&\!\! -0.23 \!\!\!&\!\! -1.19 \!\!\!&\!\! -0.06 \!\!\!&\!\! -2.52\!\!
\end{array}\!\right]}\!, \
B \!\!=\!\! {\scriptsize\left[\!\begin{array}{rrrrr}
   \!\!\! 0.81 \!\!&\!\! -0.79 \!\!&\!\!  0.00 \!\!&\!\!  0.00 \!\!&\!\!\! -0.95\!\!\\
   \!\!\!-0.34 \!\!&\!\! -0.50 \!\!&\!\!  0.06 \!\!&\!\!  0.22 \!\!&\!\!\!  0.92\!\!\\
   \!\!\!-1.32 \!\!&\!\!  1.55 \!\!&\!\! -1.22 \!\!&\!\! -0.77 \!\!&\!\!\! -1.14\!\!\\
   \!\!\!-2.11 \!\!&\!\!  0.32 \!\!&\!\!  0.00 \!\!&\!\! -0.83 \!\!&\!\!\!  0.59\!\!\\
   \!\!\! 0.31 \!\!&\!\! -0.19 \!\!&\!\! -1.09 \!\!&\!\!  0.00 \!\!&\!\!\!  0.00\!\!
\end{array}\!\right]}, \
C \!\!=\!\! {\scriptsize\begin{bmatrix}
    0.00 \!\!&\!\!  0.00 \!\!&\!\!  0.16 \!\!&\!  0.00 \!\!\!&\!\! -1.78\!\\
    1.23 \!\!&\!\! -0.38 \!\!&\!\!  0.75 \!\!&\! -0.38 \!\!\!&\!\! -0.00\!\\
    0.46 \!\!&\!\!  0.00 \!\!&\!\! -0.05 \!\!&\!  0.00 \!\!\!&\!\!  0.00\!\\
    0.00 \!\!&\!\! -0.12 \!\!&\!\! 0.23 \!\!&\! -0.12 \!\!\!&\!\!  1.14\!
\end{bmatrix}}\!.
\end{align*}
The weighting parameter $\sigma$ is chosen by $\sigma = 3$. Algorithm \ref{alg:A1} is terminated if one of the following conditions is satisfied:
\begin{itemize}
\item subproblem \eqref{eq:psd_convex_subprob} encounters a numerical problem;
\item $\norm{\Delta x^k}_{\infty}/(\norm{x^k}_{\infty}+1)\leq 10^{-3}$;
\item the maximum number of iterations, $K_{\max}$, reaches;
\item or the objective function is not significantly improved after two successive iterations (i.e. $|f^{k+1}-f^k| \leq 10^{-4}(1+|f^k|)$, for some $k=\bar{k}$
and $k=\bar{k}+1$, where $f^k :=
f(x^k)$).
\end{itemize}
In this example, Algorithm \ref{alg:A1} is terminated after $15$ iterations, whereas the objective function is not significantly improved.
However, after the $2^{\textrm{nd}}$ iteration, matrix $F$ only has $3$ nonzero elements, while the decay rate $\alpha$ is $1.17316$.
This value is much higher than the one reported in \cite{Hassibi1999}, $\alpha = 0.3543$ after $6$ iterations. We obtain the gain matrix $F$ as
\begin{equation*}
F = {\footnotesize\begin{bmatrix}
 0.6540 &  0.0000 & 0.0000 & 0.0000 \\
 0.0000 & -0.4872 & 0.0000 & 0.0000 \\
 0.0000 &  0.0000 & 0.0000 & 0.0000 \\
 0.0000 &  0.0000 & 0.0000 & 0.0000 \\
 0.0000 &  1.1280 & 0.0000 & 0.0000
\end{bmatrix}.}
\end{equation*}
With this matrix, the maximum real part of the eigenvalues of the closed-loop matrix in (\ref{eq:ss_LTI}), $A_F := A + BFC$, is $\alpha_0(A_F) := -1.40706$.
Simultaneously,
$\alpha_0(A_F^TP+PA_F+2\alpha P) = -0.327258 < 0$ and $\alpha_0(P) = 0.587574 > 0$. Note
that $\alpha_0(A_F) \neq -\alpha$ due to the in-activeness of the BMI constraint in \eqref{eq:BMI_p1} at the 2nd iteration step.

\subsection{Spectral abscissa and pseudo-spectral abscissa optimization}\label{subsec:SOF_design}
One popular problem in control theory is to optimize the spectral abscissa of the closed-loop system $\dot{x} = (A+BFC)x$.
Briefly, this problem is presented as an unconstrained optimization problem of the form:
\begin{equation}\label{eq:pure_spectral_abscissa_opt}
\min_{F\in\mathbb{R}^{n_u\times n_y}} \alpha_0(A+BFC),
\end{equation}
where $\alpha_0(A+BFC) := \sup\left\{ \textrm{Re}(\lambda) ~|~ \lambda\in\Lambda(A+BFC)\right\}$ is the spectral abscissa of $A+BFC$, $\textrm{Re}(\lambda)$
denotes the real part of
$\lambda\in\mathbb{C}$ and  $\Lambda(A+BFC)$ is the spectrum of $A+BFC$. Problem \eqref{eq:pure_spectral_abscissa_opt} has many drawbacks in terms of numerical
solution due to the nonsmoothness and
non-Lipschitz continuity of the objective function $\alpha_0$ \cite{Burke2002}.

In order to apply the method developed in this paper, we reformulate problem \eqref{eq:pure_spectral_abscissa_opt} as an optimization problem with BMI
constraints \cite{Burke2002,Leibfritz2004}:
\begin{equation}\label{eq:BMI_p2}
\left\{\begin{array}{cl}
\displaystyle
\max_{P,F,\beta} & \beta\\
\textrm{s.t.}~ &(A+BFC)^TP + P(A+BFC) + 2\beta P \prec 0, ~~P = P^T, ~P \succ 0.
\end{array}\right.
\end{equation}
Here, matrices $A\in\mathbb{R}^{n\times n}$, $B\in\mathbb{R}^{n\times n_u}$, $C\in\mathbb{R}^{n_y\times n}$ are given. Matrices $P\in\mathbb{R}^{n\times n}$ and
$F\in\mathbb{R}^{n_u\times n_y}$ and
the scalar $\beta$ are considered as variables.
If the optimal value of \eqref{eq:BMI_p2} is strictly positive then the closed-loop feedback controller $u = Fy$ stabilizes the linear system $\dot{x} =
(A+BFC)x$.

Problem \eqref{eq:BMI_p2} is very similar to \eqref{eq:BMI_p1}. Therefore, using the same trick as in \eqref{eq:BMI_p1}, we can reformulate \eqref{eq:BMI_p2} in
the form of \eqref{eq:MDC_program}.
More precisely, if we define $B_F := A + BFC + \beta I$ then the bilinear matrix mapping $A_F^TP + PA_F$ can be represented as a psd-convex-concave
decomposition of the form \eqref{eq:BMI_p1_dc2} and
problem \eqref{eq:BMI_p2} can be rewritten in the form of \eqref{eq:MDC_program}.
We implement Algorithm \ref{alg:A1} for solving this resulting problem using the same parameters and the stopping criterions as in Subsection
\ref{subsec:sparse_SOF}. In addition, we regularize the
objective function by adding the term $\frac{\rho_F}{2}\norm{F-F^k}_F^2 + \frac{\rho_P}{2}\norm{P-P^k}_F^2$, with $\rho_F
= \rho_P = 10^{-2}$. The maximum number of iterations $K_{\max}$ is set to $150$.

We test for several problems in COMP$\textrm{l}_\textrm{e}$ib and compare our results with the ones reported by HIFOO, PENBMI and LMIRank.
For LMIRank, we implement the algorithm proposed in \cite{Orsi2006}. We initialize the value of the decay rate $\alpha^0$ at $10^{-4}$ and perform an iterative
loop to increase $\alpha$ as
$\alpha^{k+1} := \alpha^k + 0.1$. The algorithm is terminated if either the problems (12) or (21) in \cite{Orsi2006} with a correspondence $\alpha$ can not be
solved or the maximum number of
iterations $K_{\max} = 100$ is reached.
\begin{table}[!ht]
\begin{center}
\renewcommand{\arraystretch}{0.7}
\begin{scriptsize}
\caption{Computational results for \eqref{eq:BMI_p2} in COMP$\textrm{l}_{\textrm{e}}$ib}\label{tb:exam2}
\newcommand{\cell}[1]{{\!}#1{\!}}
\newcommand{\cellbf}[1]{{\!}\textbf{#1}{\!}}
\begin{tabular}{|l|r|r|r|r|r|r|r|r|r|}\hline
\multicolumn{2}{|c|}{Problem} & \multicolumn{3}{c|}{Other Results, $\alpha_0(A_F)$} & \multicolumn{3}{c|}{Results and Performances} \\ \hline
\!\!\texttt{Name}\!\! &\!\!$\alpha_0(A)$\!\!&\!\!HIFOO\!\!&\!\!LMIRANK\!\!&\!\!PENBMI\!\!& \!\!$\alpha_0(A_F)$\!\! & \!\!\texttt{Iter}\!\! &
\!\!\!\texttt{time[s]}\!\! \\
\hline
\cell{AC1}  & \cell{0.000} & \cellbf{-0.2061}  & \cellbf{-8.4766} & \cellbf{-7.0758}    & \cellbf{-0.8535}  & \cell{ 41} & \cell{12.44} \\ \hline 
81
\cell{AC4}  & \cell{2.579} & \cellbf{-0.0500}  & \cellbf{-0.0500} & \cellbf{-0.0500}    & \cellbf{-0.0500}  & \cell{ 14} & \cell{4.60}  \\ \hline 
\cell{\textbf{AC5}$^{\textbf{a}}$}
            & \cell{0.999} & \cellbf{-0.7746}  & \cellbf{-1.8001} & \cellbf{-2.0438}    & \cellbf{-0.7389}  & \cell{ 28} & \cell{63.33} \\ \hline 
\cell{AC7}  & \cell{0.172} & \cellbf{-0.0322}  & \cellbf{-0.0204} & \cellbf{0.0896}     & \cellbf{-0.0673}  & \cell{150}     & \cell{111.46}\\ \hline 
2
\cell{AC8}  & \cell{0.012} & \cellbf{-0.1968}  & \cellbf{-0.4447} & \cellbf{0.4447}     & \cellbf{-0.0755}  & \cell{ 24}     & \cell{21.95} \\ \hline 
6
\cell{AC9}  & \cell{0.012} & \cellbf{-0.3389}  & \cellbf{-0.5230} & \cellbf{-0.4450}    & \cellbf{-0.3256}  & \cell{ 78}     & \cell{74.57} \\ \hline 
7
\cell{AC11} & \cell{5.451} & \cellbf{-0.0003}  & \cellbf{-5.0577} & \cellbf{-}          & \cellbf{-3.0244}  & \cell{ 61}     & \cell{38.44} \\ \hline 
52
\cell{AC12} & \cell{0.580} & \cellbf{-10.8645} & \cellbf{-9.9658} & \cellbf{-1.8757}    & \cellbf{-0.3414}  & \cell{150}     & \cell{86.72}  \\ \hline 
100
\cell{HE1}  & \cell{0.276} & \cellbf{-0.2457}  & \cellbf{-0.2071} & \cellbf{-0.2468}    & \cellbf{-0.2202}  & \cell{150}     & \cell{87.64}  \\ \hline %
32.8000, 4
\cell{HE3}  & \cell{0.087} & \cellbf{-0.4621}  & \cellbf{-2.3009} & \cellbf{-0.4063}    & \cellbf{-0.8702}  & \cell{ 47}     & \cell{48.80} \\ \hline %
700.8300, 25
\cell{HE4}  & \cell{0.234} & \cellbf{-0.7446}  & \cellbf{-1.9221} & \cellbf{-0.0909}    & \cellbf{-0.8647}  & \cell{ 63}     & \cell{71.66} \\ \hline 
21
\cell{HE5}  & \cell{0.234} & \cellbf{-0.1823}  & \cellbf{-}       & \cellbf{-0.2932}    & \cellbf{-0.0587}  & \cell{150}     & \cell{178.71}\\ \hline 
\cell{HE6}  & \cell{0.234} & \cellbf{-0.0050}  & \cellbf{-0.0050} & \cellbf{-0.0050}    & \cellbf{-0.0050}  & \cell{ 12}     & \cell{41.00} \\ \hline 
2
\cell{REA1} & \cell{1.991} & \cellbf{-16.3918} & \cellbf{-5.9736} & \cellbf{-1.7984}    & \cellbf{-3.8599}  & \cell{ 77}     & \cell{79.23} \\ \hline 
61
\cell{REA2} & \cell{2.011} & \cellbf{-7.0152}  & \cellbf{-10.0292}& \cellbf{-3.5928}    & \cellbf{-2.1778}  & \cell{ 40} & \cell{39.15} \\ \hline 
\cell{REA3} & \cell{0.000} & \cellbf{-0.0207}  & \cellbf{-0.0207} & \cellbf{-0.0207}    & \cellbf{-0.0207}  & \cell{150}& \cell{362.21} \\ \hline 
\cell{DIS2} & \cell{1.675} & \cellbf{-6.8510}  & \cellbf{-10.1207}& \cellbf{-8.3289}    & \cellbf{-8.4540}  & \cell{ 28} & \cell{37.28} \\ \hline 
\cell{DIS4} & \cell{1.442} & \cellbf{-36.7203} & \cellbf{-0.5420} & \cellbf{-92.2842}   & \cellbf{-8.0989}  & \cell{ 72} & \cell{124.23}\\ \hline 
\cell{WEC1} & \cell{0.008} & \cellbf{-8.9927}  & \cellbf{-8.7350} & \cellbf{-0.9657}    & \cellbf{-0.8779}  & \cell{150}& \cell{305.36} \\ \hline 
\cell{IH}   & \cell{0.000} & \cellbf{-0.5000}  & \cellbf{-0.5000} & \cellbf{-0.5000}    & \cellbf{-0.5000}  & \cell{  7} & \cell{39.41} \\ \hline 
\cell{CSE1} & \cell{0.000} & \cellbf{-0.4509}  & \cellbf{-0.4844} & \cellbf{-0.4490}    & \cellbf{-0.2360}  & \cell{ 38} & \cell{158.67}\\ \hline 
\cell{TF1}  & \cell{0.000} & \cellbf{-}        & \cellbf{-}       & \cellbf{-0.0618}    & \cellbf{-0.1544}  & \cell{ 56} & \cell{137.98}\\ \hline 
\cell{TF2}  & \cell{0.000} & \cellbf{-}        & \cellbf{-}       & \cellbf{-1.0e-5}    & \cellbf{-1.0e-5}  & \cell{  8} & \cell{20.41} \\ \hline 
\cell{TF3}  & \cell{0.000} & \cellbf{-}        & \cellbf{-}       & \cellbf{-0.0032}    & \cellbf{-0.0031}  & \cell{ 93} & \cell{237.93}\\ \hline 
\cell{NN1}  & \cell{3.606} & \cellbf{-3.0458}  & \cellbf{-4.4021} & \cellbf{-4.3358}    & \cellbf{-0.8746}  & \cell{ 12} & \cell{37.53}\\ \hline 
53.1700, 65
\cell{\textbf{NN5}$^{\textbf{a}}$}
            & \cell{0.420} & \cellbf{-0.0942}  & \cellbf{-0.0057} & \cellbf{-0.0942}    & \cellbf{-0.0913}  & \cell{ 11} & \cell{42.62} \\ \hline 
\cell{NN9}  & \cell{3.281} & \cellbf{-2.0789}  & \cellbf{-0.7048} & \cellbf{-}          & \cellbf{-0.0279}  & \cell{ 33} & \cell{111.41}\\ \hline
\cell{NN13} & \cell{1.945} & \cellbf{-3.2513}  & \cellbf{-4.5310} & \cellbf{-9.0741}    & \cellbf{-3.4318}  & \cell{150}& \cell{572.50} \\ \hline 
\cell{NN15} & \cell{0.000} & \cellbf{-6.9983}  & \cellbf{-11.0743}& \cellbf{-0.0278}    & \cellbf{-0.8353}  & \cell{150}& \cell{524.80} \\ \hline 
\cell{NN17} & \cell{1.170} & \cellbf{-0.6110}  & \cellbf{-0.5130} & \cellbf{-}          & \cellbf{-0.6008}  & \cell{ 99} & \cell{342.67}\\ \hline 
\end{tabular}
\end{scriptsize}
\end{center}
\vskip -0.4cm
\end{table}
The numerical results of four algorithms are reported in Table \ref{tb:exam2}. Here, we initialize the algorithm in HIFOO with the same initial guess $F^0 = 0$.
Since PENBMI and our methods solve the
same BMI problems, they are initialized by the same initial values for $P$, $F$ and $\beta$.

The notation in Table \ref{tb:exam2} consists of: \texttt{Name} is the name of problems, $\alpha_0(A)$, $\alpha_0(A_F)$ are the maximum real part of the
eigenvalues of the open-loop and closed-loop
matrices $A$, $A_F$, respectively; \texttt{iter} is the number of iterations, \texttt{time[s]} is the CPU time in second. The columns titled HIFOO, LMIRank and
PENBMI give the maximum real part of the
eigenvalues of the closed-loop system for a static output feedback controller computed by available software HIFOO \cite{Gumussoy09}, LMIRank \cite{Orsi2006}
and PENBMI \cite{Henrion2005},
respectively.  Our results can be found in the sixth column. The entries with a dash sign indicate that there is no feasible solution found.  Algorithm $1$
fails or makes only slow progress toward to
a local solution with $6$ problems: AC18, DIS5, PAS, NN6, NN7, NN12 in COMP$\textrm{l}_\textrm{e}$ib. Problems AC5 and NN5 are initialized with a different
matrix $F^0$ to avoid numerical problems.

Note that Algorithm \ref{alg:A1} as well as the algorithms implemented in HIFOO, LMIRank and PENBMI are local optimization methods, which only report a local
minimizer and these solutions may not be
the same.
To apply the LMIRank package for solving problem \eqref{eq:BMI_p2}, we have used a direct search procedure for finding $\alpha$.
The computational time of this procedure is very high compared with the other methods.

To conclude this subsection, we show that our method can be applied to solve the optimization problem of pseudo-spectral abscissa in static feedback controller
designs. This problem is described as
follows (see \cite{Burke2002,Leibfritz2004}):
\begin{equation}\label{eq:BMI_p3b}
\begin{array}{cl}
&\displaystyle\max_{\beta,\mu,\omega,F,P}\beta\\
& \textrm{s.t.}~ \begin{bmatrix} 2\beta P \!+\! A_F^T\!\!P \!+\! PA_F \!+\! \mu I \!-\! \omega I_z \!\!\! & \!\!\! \varepsilon P\\ \varepsilon P
&\omega I \end{bmatrix} \!\preceq \! 0, ~ P\succ 0, ~ P = P^T, ~ \mu > 0,
\end{array}
\end{equation}
where $A_F = A + BFC$ as before and $\omega\leq 0$.

Using the same notation $B_F = A + BFC + \beta I$ as in \eqref{eq:BMI_p2} and applying the statement b) of Lemma \ref{le:key_lemma}, the BMI constraint in this
problem can be transformed into a
psd-convex-concave one:
\begin{align*}\label{eq:BMI_reform}
&\begin{bmatrix}
B_F^T\!\!B_F \!+\! P^T\!\!P \!+\! (\mu \!-\! \omega)I \!\!&\!\! \epsilon P\\ \epsilon P \!\!&\!\! \omega I
\end{bmatrix}
\!-\! \begin{bmatrix}(B_F\!-\!P)^T\!\!(B_F\!-\!P) \!\!&\!\! 0\\ 0 \!\!&\!\! 0\end{bmatrix} \!\preceq \!0.
\end{align*}
If we denote the linearization of $(B_F-P)^T(B_F-P)$ at the iteration $k$ by $H_k$, i.e. $H_k = (B_F-P)^T(B_F^k-P^k) + (B_F^k -P^k)^T(B_F-P) -
(B_F^k-P^k)^T(B_F^k-P^k)$, then the linearized constraint
in the subproblem \eqref{eq:psd_convex_subprob} can be represented as an LMI thanks to Lemma \ref{le:key_lemma}:
\begin{equation*}
\begin{bmatrix}
H_k + (\omega-\mu)I & B_F^T & P & -\varepsilon P\\ B_F & I & 0 & 0\\
P & 0 & I & 0\\ -\varepsilon P & 0 & 0 & -\omega I
\end{bmatrix} \succeq 0.
\end{equation*}
Hence, Algorithm \ref{alg:A1} can be applied to solve problem \eqref{eq:BMI_p3b}.

\begin{remark}
If we define $\bar{F} := BFC$ then the bilinear matrix mapping $A_F^TP + PA_F$ can be rewritten as
\begin{align*}
A_F^TP + PA_F &= \frac{1}{2}\left[(P+\bar{F})^T(P+\bar{F}) - (P-\bar{F})^T(P-\bar{F})\right] - A^TP - PA.
\end{align*}
Using this decomposition, one can avoid the contribution of matrix $A$ on the bilinear term. Consequently, Algorithm \ref{alg:A1} may work better in some
specific problems.
\end{remark}

\subsection{$\mathcal{H}_2$ optimization: BMI formulation}\label{subsec:H2_prob}
In this subsection, we consider an optimization problem arising in the $\mathcal{H}_2$ synthesis of the linear system~\eqref{eq:LTI}.
Let us assume that $D_{12} = 0$ and $D_{21} = 0$, then this problem is formulated as the following optimization problem with BMI constraints
\cite{Leibfritz2003}.
\begin{equation}\label{eq:BMI_H2_prob}
\begin{array}{cl}
\displaystyle\min_{F, Q, X} &\textrm{trace}(X)\\
\textrm{s.t.}\!\!\!\!\!\!\!\! & ~ (A+BFC)Q + Q(A+BFC)^T + B_1B_1^T \preceq 0, ~~\begin{bmatrix}X & C_1Q\\ QC_1^T & Q \end{bmatrix}\succeq 0, ~~ Q \succ 0.
\end{array}
\end{equation}
Here, we also assume that $B_1B_1^T$ is positive definite.
Otherwise, we use $B_1B_1^T + \varepsilon I$ instead of $B_1B_1^T$ with $\varepsilon = 10^{-5}$ in \eqref{eq:BMI_H2_prob}.

In order to apply Algorithm \ref{alg:A1} for solving problem \eqref{eq:BMI_H2_prob}, a starting point $x^0\in\textrm{ri}(\mathcal{D})$ is required. This task
can be done by performing some extra steps
called \textit{Phase 1}. The algorithm is now split in two phases as follows.

\noindent\textbf{Phase 1: }(\textit{Determine a starting point $x^0$}).
\begin{itemize}
\item[]\textit{Step 1.} If $\alpha_0(A + A^T) < 0$ then we set $F^0 := 0$. Otherwise, go to
Step 3.
\item[]\textit{Step 2.} Solve the following optimization problem with LMI constraints:
\begin{equation}\label{eq:LMI_H2_prob}
\begin{array}{cl}
\displaystyle\min_{Q, X} \textrm{trace}(X) ~~~~
\textrm{s.t.} ~~ A_{F^0}Q + QA_{F^0}^T + B_1B_1^T \prec 0, ~~\begin{bmatrix}X & C_1Q\\ QC_1^T & Q \end{bmatrix}\succ 0, ~~Q \succ 0,
\end{array}
\end{equation}
where $A_{F^0} := A + BF^0C$. If this problem has a solution $Q^0$ and $X^0$ then terminate Phase 1 and using $F^0$ together with $Q^0, X^0$ as a starting point
$x^0$ for Phase 2. Otherwise, go to
Step 3.
\item[]\textit{Step 3.} Solve the following feasibility problem with LMI constraints:
\begin{equation*}\label{eq:LMI_H2_phase1}
\begin{array}{cl}
\textrm{Find}& \textrm{$P\succ 0$ and $K$ such that:}\\
&\begin{bmatrix}PA+A^TP+KC+C^TK^T & PB_1\\ B_1^TP &-\sigma_0^2I_{w}\end{bmatrix} \preceq 0, ~~\begin{bmatrix}X & C_1\\ C_1^T & P\end{bmatrix} \succeq 0,
\end{array}
\end{equation*}
to obtain $K^{*}$ and $P^{*}$, where $\sigma_0$ is a given regularization factor.
Compute $F^{*} := B^{+}(P^{*})^{-1}K^{*}$, where $B^{+}$ is a pseudo-inverse of $B$, and resolve problem \eqref{eq:LMI_H2_prob} with $F^0 := F^{*}$.
If problem \eqref{eq:LMI_H2_prob} has a solution $Q^0$ and $X^0$ then terminate Phase 1 and set $x^0 := (F^0, Q^0, X^0)$.
Otherwise, perform Step 4.
\item[]\textit{Step 4.} Apply the method in Subsection~\ref{subsec:SOF_design} to solve the following BMI feasibility problem:
\begin{align}\label{eq:BMI_H2_phase1}
\textrm{Find }& \textrm{$F$ and $Q\succ 0$ such that}: ~~ (A+BFC)Q + Q(A+BFC)^T + B_1B_1^T \prec 0.
\end{align}
If this problem has a solution $F^0$ then go back to Step 2. Otherwise, declare that no strictly feasible point is found.
\end{itemize}
\noindent\textbf{Phase 2: }(\textit{Solve problem \eqref{eq:BMI_H2_prob}}). Perform Algorithm \ref{alg:A1} with the starting point $x^0$ found at Phase 1.
\eofproof

Note that Step 3 of Phase 1 corresponds to determining a full state feedback controller and approximating it subsequently with an output feedback controller.
Step 4 of Phase 1 is usually expensive.
Therefore, in our numerical implementation, we terminate Step 4 after finding a point such that $\alpha_0((A+BFC)Q + Q(A+BFC)^T + B_1B_1^T) \leq -0.1$.

\begin{remark}\label{re:stop_phase1}
The algorithm described in Phase 1 is finite. It terminates either at Step 4 if no feasible point is found or at Step 2 if a feasible point is found.
Indeed, if a feasible matrix $F^0$ is found at Step 4, the first BMI constraint of \eqref{eq:LMI_H2_prob} is feasible with some $Q \succ 0$. Thus we can find an
appropriate matrix $X$ such that $X -
CQC^T \succ 0$, which implies the second LMI constraint of \eqref{eq:LMI_H2_prob} is satisfied. Consequently, problem \eqref{eq:LMI_H2_prob} has a solution.
\end{remark}

The method used in Phase 1 is heuristic. It can be improved when we apply to a certain problem.
However, as we can see in the numerical results, it performs quite acceptable for majority of the test problems.

In the following numerical examples, we implement Phase 1 and Phase 2 of the algorithm using the decomposition
\begin{align*}
A_FQ + QA_F^T + B_1B_1^T &= \frac{1}{2}(A_F+Q)(A_F+Q)^T  + B_1B_1^T - \frac{1}{2}(A_F-Q)(A_F-Q)^T
\end{align*}
for the BMI form at the left-top corner of the first constraint in \eqref{eq:BMI_H2_prob}.
The regularization parameters and the stopping criterion for Algorithm \ref{alg:A1} are chosen as in Subsection \ref{subsec:sparse_SOF} and $K_{\max} = 300$.
\begin{table}[!htbp]
\begin{center}
\begin{scriptsize}
\renewcommand{\arraystretch}{0.7}
\caption{$\mathcal{H}_2$ synthesis benchmarks on COMP$\textrm{l}_{\textrm{e}}$ib plants}\label{tb:H2_prob}
\newcommand{\cell}[1]{{\!}#1{\!}}
\newcommand{\cellbf}[1]{{\!}\textbf{#1}{\!}}
\begin{tabular}{|l|c|c|c|c|c|c|r|r|r|r|r|}\hline
\multicolumn{6}{|c|}{Problem} & \multicolumn{2}{c|}{Other Results, $\mathcal{H}_2$} & \multicolumn{3}{|c|}{\!\!{Results and Performances  }\!\!} \\ \hline
\texttt{Name}\!\!&\!\!\!\!$n_x$\!\!\!\!&\!\!\!\!$n_y$\!\!\!\!&\!\!\!\!$n_u$\!\!\!\!& \!\!\!\!$n_z$\!\!\!\! & \!\!\!\!$n_w$\!\!\!\!& \!\!HIFOO~~~\!\! &
\!\!PENBMI~~~\!\! &
\!\!\!$\mathcal{H}_2$~~~\!\!\!&
\!\!\texttt{iter}\!\! & \!\texttt{time\![s]}~~~\!\! \\ \hline
\cell{AC1$^{b}$}    & \cell{5}  & \cell{3} & \cell{3} & \cell{2} & \cell{3} & \cellbf{0.0250} & \cellbf{0.0061} &  \cellbf{0.0540} & \cell{3}   & \cell{3.380}
\\ \hline
\cell{AC2$^{b}$}    & \cell{5}  & \cell{3} & \cell{3} & \cell{5} & \cell{3} & \cellbf{0.0257} & \cellbf{0.0075} &  \cellbf{0.0540} & \cell{3}   & \cell{3.390}
\\ \hline
\cell{AC3}          & \cell{5}  & \cell{4} & \cell{2} & \cell{5} & \cell{5} & \cellbf{2.0964} & \cellbf{2.0823} &  \cellbf{2.1117} & \cell{210} & \cell{73.380}
\\ \hline
\cell{AC4}          & \cell{4}  & \cell{2} & \cell{1} & \cell{2} & \cell{2} & \cellbf{11.0269}& \cellbf{-}      &  \cellbf{11.0269}& \cell{2}   & \cell{0.990}
\\ \hline
\cell{AC6}          & \cell{7}  & \cell{4} & \cell{2} & \cell{7} & \cell{7} & \cellbf{2.8648} & \cellbf{2.8648} &  \cellbf{2.8664} & \cell{153} & \cell{124.230}
\\ \hline
\cell{AC7}          & \cell{9}  & \cell{2} & \cell{1} & \cell{1} & \cell{4} & \cellbf{0.0172} & \cellbf{0.0162} &  \cellbf{0.0176} & \cell{1}   & \cell{0.670}
\\ \hline
\cell{AC8}          & \cell{9}  & \cell{5} & \cell{1} & \cell{2} & \cell{10}& \cellbf{0.6330} & \cellbf{0.7403} &  \cellbf{0.6395} & \cell{300} & \cell{282.760}
\\ \hline
\cell{AC12$^{b}$}   & \cell{4}  & \cell{4} & \cell{3} & \cell{1} & \cell{3} & \cellbf{0.0022} & \cellbf{0.0106} &  \cellbf{0.0992} & \cell{36}  & \cell{28.540}
\\ \hline
\cell{AC15$^{b}$}   & \cell{4}  & \cell{3} & \cell{2} & \cell{6} & \cell{4} & \cellbf{1.5458} & \cellbf{1.4811} &  \cellbf{1.5181} & \cell{264} & \cell{85.390}
\\ \hline
\cell{AC16$^{b}$}   & \cell{4}  & \cell{4} & \cell{2} & \cell{6} & \cell{4} & \cellbf{1.4769} & \cellbf{1.4016} &  \cellbf{1.4427} & \cell{300} & \cell{99.790}
\\ \hline
\cell{AC17}         & \cell{4}  & \cell{2} & \cell{1} & \cell{4} & \cell{4} & \cellbf{1.5364} & \cellbf{1.5347} &  \cellbf{1.5507} & \cell{171} & \cell{49.350}
\\ \hline
\cell{HE2}          & \cell{4}  & \cell{2} & \cell{2} & \cell{4} & \cell{4} & \cellbf{3.4362} & \cellbf{3.4362} &  \cellbf{4.7406} & \cell{263} & \cell{97.310}
\\ \hline
\cell{HE3$^{b}$}    & \cell{8}  & \cell{6} & \cell{4} & \cell{10}& \cell{1} & \cellbf{0.0197} & \cellbf{0.0071} &  \cellbf{0.1596} & \cell{249} & \cell{217.360}
\\ \hline
\cell{HE4$^{b}$}    & \cell{8}  & \cell{6} & \cell{4} & \cell{12}& \cell{8} & \cellbf{6.6436} & \cellbf{6.5785} &  \cellbf{7.1242} & \cell{300} & \cell{412.830}
\\ \hline
\cell{REA1}         & \cell{4}  & \cell{3} & \cell{2} & \cell{4} & \cell{4} & \cellbf{0.9442} & \cellbf{0.9422} &  \cellbf{1.0622} & \cell{249} & \cell{80.810}
\\ \hline
\cell{REA2$^{b}$}   & \cell{4}  & \cell{2} & \cell{2} & \cell{4} & \cell{4} & \cellbf{1.0339} & \cellbf{1.0229} &  \cellbf{1.1989} & \cell{300} & \cell{101.730}
\\ \hline
\cell{DIS1}         & \cell{8}  & \cell{4} & \cell{4} & \cell{8} & \cell{1} & \cellbf{0.6705} & \cellbf{0.1174} &  \cellbf{0.7427} & \cell{300} & \cell{255.810}
\\ \hline
\cell{DIS2}         & \cell{3}  & \cell{2} & \cell{2} & \cell{3} & \cell{3} & \cellbf{0.4013} & \cellbf{0.3700} &  \cellbf{0.3819} & \cell{4}   & \cell{1.370}
\\ \hline
\cell{DIS3}         & \cell{6}  & \cell{4} & \cell{4} & \cell{6} & \cell{6} & \cellbf{0.9527} & \cellbf{0.9434} &  \cellbf{1.0322} & \cell{300} & \cell{210.470}
\\ \hline
\cell{DIS4}         & \cell{6}  & \cell{6} & \cell{4} & \cell{6} & \cell{6} & \cellbf{1.0117} & \cellbf{0.9696} &  \cellbf{1.0276} & \cell{300} & \cell{210.690}
\\ \hline
\cell{WEC1$^{b}$}   & \cell{10} & \cell{4} & \cell{3} & \cell{10}& \cell{10}& \cellbf{7.3940} & \cellbf{8.1032} &  \cellbf{12.9093}& \cell{119} & \cell{190.150}
\\ \hline
\cell{WEC2$^{b}$}   & \cell{10} & \cell{4} & \cell{3} & \cell{10}& \cell{10}& \cellbf{6.7908} & \cellbf{7.6502} &  \cellbf{12.2102}& \cell{261} & \cell{407.470}
\\ \hline
\cell{AGS}          & \cell{12} & \cell{2} & \cell{2} & \cell{12}& \cell{12}& \cellbf{6.9737} & \cellbf{6.9737} &  \cellbf{6.9838} & \cell{18}  & \cell{28.830}
\\ \hline
\cell{BDT1}         & \cell{11} & \cell{3} & \cell{3} & \cell{6} & \cell{1} & \cellbf{0.0024} & \cellbf{-} &  \cellbf{0.0017} & \cell{51}  & \cell{64.410} \\
\hline
\cell{MFP}          & \cell{4}  & \cell{2} & \cell{3} & \cell{4} & \cell{4} & \cellbf{6.9724} & \cellbf{6.9724} &  \cellbf{7.0354} & \cell{300} & \cell{114.560}
\\ \hline
\cell{PSM}          & \cell{7}  & \cell{3} & \cell{2} & \cell{5} & \cell{2} & \cellbf{0.0330} & \cellbf{0.0007} &  \cellbf{0.1753} & \cell{300} & \cell{217.250}
\\ \hline
\cell{EB2$^{b}$}    & \cell{10} & \cell{1} & \cell{1} & \cell{2} & \cell{2} & \cellbf{0.0640} & \cellbf{0.0084} &  \cellbf{0.1604} & \cell{114} & \cell{131.380}
\\ \hline
\cell{EB3}          & \cell{10} & \cell{1} & \cell{1} & \cell{2} & \cell{2} & \cellbf{0.0732} & \cellbf{0.0072} &  \cellbf{0.0079} & \cell{7}   & \cell{6.240}
\\ \hline
\cell{TF1$^{b}$}    & \cell{7}  & \cell{4} & \cell{2} & \cell{4} & \cell{1} & \cellbf{0.0945} & \cellbf{-} &  \cellbf{0.1599} & \cell{192} & \cell{166.810} \\
\hline
\cell{TF2}          & \cell{7}  & \cell{3} & \cell{2} & \cell{4} & \cell{1} & \cellbf{11.1803}& \cellbf{-} &  \cellbf{11.1803}& \cell{3}   & \cell{2.810} \\
\hline
\cell{TF3$^{b}$}    & \cell{7}  & \cell{3} & \cell{2} & \cell{4} & \cell{1} & \cellbf{0.1943} & \cellbf{0.1424} &  \cellbf{0.2565} & \cell{138} & \cell{128.250}
\\ \hline
\cell{NN2}          & \cell{2}  & \cell{1} & \cell{1} & \cell{2} & \cell{2} & \cellbf{1.1892} & \cellbf{1.1892} &  \cellbf{1.1892} & \cell{4}   & \cell{1.090}
\\ \hline
\cell{NN4}          & \cell{4}  & \cell{3} & \cell{2} & \cell{4} & \cell{4} & \cellbf{1.8341} & \cellbf{1.8335} &  \cellbf{1.8590} & \cell{222} & \cell{67.260}
\\ \hline
\cell{NN8}          & \cell{3}  & \cell{2} & \cell{2} & \cell{3} & \cell{3} & \cellbf{1.5152} & \cellbf{1.5117} &  \cellbf{1.5725} & \cell{183} & \cell{50.170}
\\ \hline
\cell{NN11}         & \cell{16} & \cell{5} & \cell{3} & \cell{3} & \cell{3} & \cellbf{0.1178} & \cellbf{0.0790} &  \cellbf{0.1263} & \cell{39}  & \cell{91.070}
\\ \hline
\cell{NN13$^{b}$}   & \cell{6}  & \cell{2} & \cell{2} & \cell{3} & \cell{3} & \cellbf{26.1012}& \cellbf{26.1314} &  \cellbf{62.3995}& \cell{138} &
\cell{112.750} \\ \hline
\cell{NN14$^{b}$}   & \cell{6}  & \cell{2} & \cell{2} & \cell{3} & \cell{3} & \cellbf{26.1448}& \cellbf{26.1314} &  \cellbf{62.3995}& \cell{138} &
\cell{110.020} \\ \hline
\cell{NN15}         & \cell{3}  & \cell{2} & \cell{2} & \cell{4} & \cell{1} & \cellbf{0.0245} & \cellbf{-} &  \cellbf{0.0210} & \cell{6}   & \cell{2.060} \\
\hline
\cell{NN16$^{b}$}   & \cell{8}  & \cell{4} & \cell{4} & \cell{4} & \cell{8} & \cellbf{0.1195} & \cellbf{0.1195} &  \cellbf{0.1195} & \cell{3}   & \cell{23.030}
\\ \hline
\cell{NN17}         & \cell{3}  & \cell{1} & \cell{2} & \cell{2} & \cell{1} & \cellbf{3.2530} & \cellbf{3.2404} &  \cellbf{3.3329} & \cell{300} & \cell{88.770}
\\ \hline
\end{tabular}
\end{scriptsize}
\end{center}
\vskip -0.4cm
\end{table}
We test the algorithm for many problems in COMP$\textrm{l}_{\textrm{e}}$ib and the computational results are reported in Table \ref{tb:H2_prob}.
For the comparison purpose, we also carry out the test with HIFOO \cite{Gumussoy09} and PENBMI \cite{Henrion2005}, and the results are put in the columns marked
by HIFOO and PENBMI in Table
\ref{tb:H2_prob}, respectively. The initial controller for HIFOO is set to $F^0$ and the BMI parameters for PENBMI are initialized with $(F, Q, X)=(F^0, Q^0,
X^0)$.
Here, $n, n_y, n_z, n_w, n_u$ are the dimensions of problems, the columns titled HIFOO and PENBMI give the $\mathcal{H}_2$ norm of the closed-loop system for
the static output feedback controller
computed by HIFOO and PENBMI; \texttt{iter} and \texttt{time[s]} are the number of iterations and CPU time in second of Algorithm \ref{alg:A1}, respectively,
included Phase 1 and Phase 2. Problems
marked by
``b'' mean that Step 4 in Phase 1 is performed. In Table \ref{tb:H2_prob}, we only report the problems that were solved by Algorithm \ref{alg:A1}. The numerical
results allow us to conclude that
Algorithm \ref{alg:A1},
PENBMI and HIFOO report similar values for majority of the test problems in COMP$\textrm{l}_{\textrm{e}}$ib.

If $D_{12} \neq 0$ then the second LMI constraint of \eqref{eq:BMI_H2_prob} becomes a BMI constraint:
\begin{equation}\label{eq:BMI_H2_const2}
\begin{bmatrix} X & (C_1 + D_{12}FC)Q \\ Q(C_1+D_{12}FC)^T & Q\end{bmatrix} \succeq 0,
\end{equation}
which is equivalent to $X - C_FQC_F^T \succeq 0$, where $C_F := C_1 + D_{12}FC$.
Since $f(Q) := Q^{-1}$ is convex on $\mathbf{S}^{n_x}_{++}$ (see Lemma \ref{le:key_lemma_a} a)), this BMI constraint can be reformulated as a convex-concave
matrix inequality
constraint of
the form:
\begin{equation}\label{eq:BMI_H2_const2_DC}
\begin{bmatrix} X & C_F \\ C_F^T & O\end{bmatrix} + \begin{bmatrix}O & O\\ O & Q^{-1}\end{bmatrix}\succeq 0.
\end{equation}
By linearizing the concave term $-f(Q)$ at $Q = Q^k$ as $(Q^k)^{-1} - (Q^{k})^{-1}(Q-Q^k)(Q^k)^{-1}$ (see \cite{Boyd2004}), the resulting constraint can be
written as an LMI constraint. Therefore,
Algorithm \ref{alg:A1} can be applied to solve problem \eqref{eq:BMI_H2_const2} in the case $D_{12}\neq 0$.

\subsection{$\mathcal{H}_{\infty}$ optimization: BMI formulation}\label{subsec:Hinf_prob}
Alternatively, we can also apply Algorithm \ref{alg:A1} to solve the optimization with BMI constraints arising in $\mathcal{H}_{\infty}$ optimization of the
linear system \eqref{eq:LTI}.
Let us assume that $D_{21} = 0$, then this problem is reformulated as the following optimization problem with BMI constraints \cite{Leibfritz2003}:
\begin{equation}\label{eq:BMI_Hinf_prob}
\begin{array}{cl}
\displaystyle\min_{F, X, \gamma} &\gamma \\
\textrm{s.t.}\!\!\!\!\!\!\!\! & ~\begin{bmatrix}A_F^TX + XA_F &  XB_1 & C_F^T\\ B_1^TX & -\gamma I_w & D_{11}^T\\ C_F & D_{11} & -\gamma I_z\end{bmatrix} \prec
0, ~~ X \succ 0, ~ \gamma > 0.
\end{array}
\end{equation}
Here, as before, $A_F = A + BFC$ and $C_F = C_1 + D_{12}FC$.
The bilinear matrix term $A_F^TX + XA_F$ at the top-corner of the last constraint can be decomposed as \eqref{eq:BMI_p1_dc1} or \eqref{eq:BMI_p1_dc2}.
Therefore, we can use these decompositions to
transform problem
\eqref{eq:BMI_Hinf_prob} into \eqref{eq:MDC_program}. After linearization, the resulting subproblem is also rewritten as a standard SDP problem by applying
Lemma \ref{le:key_lemma}. We omit this
specification here.

To determine a starting point, we perform Phase 1 which is similar to the one carried out in the $\h_2$-optimization subsection.

\noindent\textbf{Phase 1.} (\textit{Determine a starting point $x^0\in\textrm{ri}(\mathcal{D})$}).
\begin{itemize}
\item[]\textit{Step 1. } If $\alpha_0(A^T+A) < 0$ then set $F^0 = 0$. Otherwise, go to Step 3.
\item[]\textit{Step 2. } Solve the following optimization with LMI constraints
\begin{equation}\label{eq:LMI_Hinf_prob}
\begin{array}{cl}
\displaystyle\min_{\gamma, X} &\gamma ~~~ \textrm{s.t.} ~~~ 
\begin{bmatrix}A_{F^0}^TX + XA_{F^0} & XB_1 & C_{F^0}^T\\ B_1^TX & -\gamma I_w & D_{11}^T\\ C_{F^0} & D_{11} & -\gamma I_z \end{bmatrix} \!\prec 0, ~~X \succ 0,
~\gamma > 0,
\end{array}
\end{equation}
where $A_{F^0} := A + BF^0C$ and $C_{F^0} := C_1 + D_{12}F^0C$. If this problem has a solution $\gamma^0$ and $X^0$ then terminate Phase 1 and using $F^0$
together with $\gamma^0, X^0$ as a starting
point $x^0$ for Phase 2. Otherwise, go to Step 3.
\item[]\textit{Step 3.} Solve the following feasibility problem of LMI constraints:
\begin{equation*}\label{eq:LMI_Hinf_phase1}
\begin{array}{cl}
\textrm{Find $P\succ 0$, $\gamma>0$ and $K$ such that:} \begin{bmatrix}\! PA^T\!\! + \!AP \!+\! K^TB^T\!\! +\! BK \!\!& B_1& \!\! PC_1 \!+\! K^T\!D_{12}^T\!\!
\\ B_1^T \!&\! -\gamma
I_{w} \!&\! D_{11}^T\!\! \\C_1P \!+\! D_{12}K \!&\! D_{11} \!&\!
-\gamma I_z\!\!\end{bmatrix} \!\prec\! 0,\\
\end{array}
\end{equation*}
to obtain $K^{*}$, $\gamma^{*}$ and $P^{*}$.
Compute $F^{*} := K^{*}(P^{*})^{-1}C^{+}$, where $C^{+}$ is a pseudo-inverse of $C$, and resolve problem \eqref{eq:LMI_Hinf_prob} with $F^0 := F^{*}$.
If problem \eqref{eq:LMI_Hinf_prob} has a solution $X^0$ and $\gamma^0$ then terminate Phase 1. Set $x^0 := (F^0, X^0,\gamma^0)$.
Otherwise, perform Step 4.
\item[]\textit{Step 4.} Apply the method in Subsection \ref{subsec:SOF_design} to solve the following BMI feasibility problem:
\begin{align}\label{eq:BMI_Hinf_phase1}
\textrm{Find } ~\textrm{$F$ and $P\succ 0$ such that}: ~~ (A+BFC)^TP + P(A+BFC) \prec 0.
\end{align}
If this problem has a solution $F^0$ then go back to Step 2. Otherwise, declare that no strictly feasible point for \eqref{eq:BMI_Hinf_prob} is found.
{~}\eofproof
\end{itemize}
As in the $\mathcal{H}_2$ problem, Phase 1 of the $\mathcal{H}_{\infty}$ is also terminated after finitely many iterations.
In this subsection, we also test this algorithm for several problems in COMP$\textrm{l}_{\textrm{e}}$ib using the same parameters and the stopping criterion as
in the previous subsection. The
computational results are shown in Table \ref{tb:H_inf_problems}.
\begin{table}[!ht]
\begin{center}
\begin{scriptsize}
\renewcommand{\arraystretch}{0.7}
\caption{$\mathcal{H}_{\infty}$ synthesis benchmarks on  COMP$\textrm{l}_{\textrm{e}}$ib plants}\label{tb:H_inf_problems}
\newcommand{\cell}[1]{{\!}#1{\!}}
\newcommand{\cellbf}[1]{{\!}\textbf{#1}{\!}}
\begin{tabular}{|l|c|c|c|c|c|r|r|r|r|r|r|}\hline
\multicolumn{6}{|c|}{Problem} & \multicolumn{2}{c|}{Other Results, $\mathcal{H}_\infty$} & \multicolumn{3}{|c|}{\!\!{Results and Performances  }\!\!} \\ \hline
\texttt{Name}\!\! & \!\!\!\!$n_x$\!\!\!\! &\!\!\!\!$n_y$\!\!\!\! &\!\!\!\!$n_u$\!\!\!\!& \!\!\!\!$n_z$\!\!\!\! & \!\!\!\!$n_w$\!\!\!\!& \!\!HIFOO\!\! &
\!\!PENBMI~~~~\!\! &
\!\!\!$\mathcal{H}_\infty$~~\!\!\! &
\!\!\texttt{iter}\!\! & \!\texttt{time\![s]}~~~\!\! \\ \hline
\cell{AC1}       & \cell{5} & \cell{3} & \cell{3} & \cell{2} & \cell{3} & \cellbf{0.0000} & \cellbf{-} & \cellbf{0.0177} & \cell{93}   & \cell{28.050} \\ \hline
\cell{AC2}       & \cell{5} & \cell{3} & \cell{3} & \cell{5} & \cell{3} & \cellbf{0.1115} & \cellbf{-} & \cellbf{0.1140} & \cell{99}   & \cell{32.540} \\ \hline
\cell{AC3}       & \cell{5} & \cell{4} & \cell{2} & \cell{5} & \cell{5} & \cellbf{4.7021} & \cellbf{-} & \cellbf{3.4859} & \cell{210}  & \cell{76.170} \\ \hline
\cell{AC4}       & \cell{4} & \cell{2} & \cell{1} & \cell{2} & \cell{2} & \cellbf{0.9355} & \cellbf{-} & \cellbf{69.9900}& \cell{2}    & \cell{2.620} \\ \hline
\cell{AC6}       & \cell{7} & \cell{4} & \cell{2} & \cell{7} & \cell{7} & \cellbf{4.1140} & \cellbf{-} & \cellbf{4.1954} & \cell{167}  & \cell{138.570} \\
\hline
\cell{AC7}       & \cell{9} & \cell{2} & \cell{1} & \cell{1} & \cell{4} & \cellbf{0.0651} & \cellbf{0.3810} & \cellbf{0.0548} & \cell{300}  & \cell{278.330} \\
\hline
\cell{AC8}       & \cell{9} & \cell{5} & \cell{1} & \cell{2} & \cell{10}& \cellbf{2.0050} & \cellbf{-} & \cellbf{3.0520} & \cell{247}  & \cell{298.070} \\
\hline
\cell{AC9$^{b}$} & \cell{10}& \cell{5} & \cell{4} & \cell{2} & \cell{10}& \cellbf{1.0048} & \cellbf{-} & \cellbf{0.9237} & \cell{300}  & \cell{470.910} \\
\hline
\cell{AC11$^{b}$}& \cell{5} & \cell{4} & \cell{2} & \cell{5} & \cell{5} & \cellbf{3.5603} & \cellbf{-} & \cellbf{3.0104} & \cell{68}   & \cell{60.260} \\ \hline
\cell{AC12}      & \cell{4} & \cell{4} & \cell{3} & \cell{1} & \cell{3} & \cellbf{0.3160} & \cellbf{-} & \cellbf{2.3025} & \cell{300}  & \cell{181.870} \\
\hline
\cell{AC15}      & \cell{4} & \cell{3} & \cell{2} & \cell{6} & \cell{4} & \cellbf{15.2074}& \cellbf{427.4106} & \cellbf{15.1995}& \cell{105}  & \cell{36.700} \\
\hline
\cell{AC16}      & \cell{4} & \cell{4} & \cell{2} & \cell{6} & \cell{4} & \cellbf{15.4969}& \cellbf{-} & \cellbf{14.9881}& \cell{186}  & \cell{68.820} \\ \hline
\cell{AC17}      & \cell{4} & \cell{2} & \cell{1} & \cell{4} & \cell{4} & \cellbf{6.6124} & \cellbf{-} & \cellbf{6.6373} & \cell{129}  & \cell{42.400} \\ \hline
\cell{HE1$^{b}$} & \cell{4} & \cell{1} & \cell{2} & \cell{2} & \cell{2} & \cellbf{0.1540} & \cellbf{1.5258} & \cellbf{0.1807} & \cell{300}  & \cell{143.520} \\
\hline
\cell{HE2}       & \cell{4} & \cell{2} & \cell{2} & \cell{4} & \cell{4} & \cellbf{4.4931} & \cellbf{-} & \cellbf{6.7846} & \cell{177}  & \cell{67.470} \\ \hline
\cell{HE3}       & \cell{8} & \cell{6} & \cell{4} & \cell{10}& \cell{1} & \cellbf{0.8545} & \cellbf{1.6843} & \cellbf{0.9243} & \cell{105}  & \cell{95.000} \\
\hline
\cell{HE4$^{b}$} & \cell{8} & \cell{6} & \cell{4} & \cell{12}& \cell{8} & \cellbf{23.3448}& \cellbf{-} & \cellbf{22.8713}& \cell{252}  & \cell{325.580} \\
\hline
\cell{HE5$^{b}$} & \cell{8} & \cell{2} & \cell{4} & \cell{4} & \cell{3} & \cellbf{8.8952} & \cellbf{-} & \cellbf{37.3906}& \cell{300}  & \cell{430.820} \\
\hline
\cell{REA1}      & \cell{4} & \cell{3} & \cell{2} & \cell{4} & \cell{4} & \cellbf{0.8975} & \cellbf{-} & \cellbf{0.8815} & \cell{96}   & \cell{34.430} \\ \hline
\cell{REA2$^{b}$}& \cell{4} & \cell{2} & \cell{2} & \cell{4} & \cell{4} & \cellbf{1.1881} & \cellbf{-} & \cellbf{1.4188} & \cell{300}  & \cell{118.320} \\
\hline
\cell{REA3}      & \cell{12}& \cell{3} & \cell{1} & \cell{12}& \cell{12}& \cellbf{74.2513}& \cellbf{74.4460} & \cellbf{74.5478}& \cell{2}    & \cell{234.800} \\
\hline
\cell{DIS1}      & \cell{8} & \cell{4} & \cell{4} & \cell{8} & \cell{1} & \cellbf{4.1716} & \cellbf{-} & \cellbf{4.1943} & \cell{93}   & \cell{66.130} \\ \hline
\cell{DIS2}      & \cell{3} & \cell{2} & \cell{2} & \cell{3} & \cell{3} & \cellbf{1.0548} & \cellbf{1.7423} & \cellbf{1.1546} & \cell{54}   & \cell{17.120} \\
\hline
\cell{DIS3}      & \cell{6} & \cell{4} & \cell{4} & \cell{6} & \cell{6} & \cellbf{1.0816} & \cellbf{-} & \cellbf{1.1382} & \cell{285}  & \cell{195.960} \\
\hline
\cell{DIS4}      & \cell{6} & \cell{6} & \cell{4} & \cell{6} & \cell{6} & \cellbf{0.7465} & \cellbf{-} & \cellbf{0.7498} & \cell{126}  & \cell{93.220} \\ \hline
\cell{TG1$^{b}$} & \cell{10}& \cell{2} & \cell{2} & \cell{10}& \cell{10}& \cellbf{12.8462}& \cellbf{-} & \cellbf{12.9336}& \cell{45}   & \cell{84.380} \\ \hline
\cell{AGS}       & \cell{12}& \cell{2} & \cell{2} & \cell{12}& \cell{12}& \cellbf{8.1732} & \cellbf{188.0315} & \cellbf{8.1732} & \cell{24}   & \cell{55.290} \\
\hline
\cell{WEC2}      & \cell{10}& \cell{4} & \cell{3} & \cell{10}& \cell{10}& \cellbf{4.2726} & \cellbf{32.9935} & \cellbf{6.6082} & \cell{300}  & \cell{476.010} \\
\hline
\cell{WEC3}      & \cell{10}& \cell{4} & \cell{3} & \cell{10}& \cell{10}& \cellbf{4.4497} & \cellbf{200.1467} & \cellbf{6.8402} & \cell{300}  & \cell{425.330}
\\ \hline
\cell{BDT1}      & \cell{11}& \cell{3} & \cell{3} & \cell{6} & \cell{1} & \cellbf{0.2664} & \cellbf{-} & \cellbf{0.8562} & \cell{29}   & \cell{40.910} \\ \hline
\cell{MFP}       & \cell{4} & \cell{2} & \cell{3} & \cell{4} & \cell{4} & \cellbf{31.5899}& \cellbf{-} & \cellbf{31.6079}& \cell{171}  & \cell{57.430} \\ \hline
\cell{IH}        & \cell{21}& \cell{10}& \cell{11}& \cell{11}& \cell{21}& \cellbf{1.9797} & \cellbf{-} & \cellbf{1.1858} & \cell{114}  & \cell{1206.340} \\
\hline
\cell{CSE1}      & \cell{20}& \cell{10}& \cell{2} & \cell{12}& \cell{1} & \cellbf{0.0201} & \cellbf{-} & \cellbf{0.0220} & \cell{3}    & \cell{20.250} \\ \hline
\cell{PSM}       & \cell{7} & \cell{3} & \cell{2} & \cell{5} & \cell{2} & \cellbf{0.9202} & \cellbf{-} & \cellbf{0.9227} & \cell{87}   & \cell{67.470} \\ \hline
\cell{EB1}       & \cell{10}& \cell{1} & \cell{1} & \cell{2} & \cell{2} & \cellbf{3.1225} & \cellbf{39.9526} & \cellbf{2.0276} & \cell{300}  & \cell{295.420} \\
\hline
\cell{EB2}       & \cell{10}& \cell{1} & \cell{1} & \cell{2} & \cell{2} & \cellbf{2.0201} & \cellbf{39.9547} & \cellbf{0.8148} & \cell{84} & \cell{73.770} \\
\hline
\cell{EB3}       & \cell{10}& \cell{1} & \cell{1} & \cell{2} & \cell{2} & \cellbf{2.0575} & \cellbf{3995311.0743} & \cellbf{0.8153} & \cell{84} & \cell{75.820}
\\ \hline
\cell{NN1}       & \cell{3} & \cell{2} & \cell{1} & \cell{3} & \cell{3} & \cellbf{13.9782}& \cellbf{14.6882} & \cellbf{18.4813}& \cell{300} & \cell{144.940} \\
\hline
\cell{NN2}       & \cell{2} & \cell{1} & \cell{1} & \cell{2} & \cell{2} & \cellbf{2.2216} & \cellbf{-} & \cellbf{2.2216} & \cell{9} & \cell{4.060} \\ \hline
\cell{NN4}       & \cell{4} & \cell{3} & \cell{2} & \cell{4} & \cell{4} & \cellbf{1.3627} & \cellbf{-} & \cellbf{1.3802} & \cell{156} & \cell{51.480} \\ \hline
\cell{NN8}       & \cell{3} & \cell{2} & \cell{2} & \cell{3} & \cell{3} & \cellbf{2.8871} & \cellbf{78281181.1490} & \cellbf{2.9345} & \cell{180} &
\cell{51.830} \\ \hline
\cell{NN9$^{b}$} & \cell{5} & \cell{2} & \cell{3} & \cell{4} & \cell{2} & \cellbf{28.9083}& \cellbf{-} & \cellbf{32.1222}& \cell{300} & \cell{129.920} \\ \hline
\cell{NN11$^{b}$}& \cell{16}& \cell{5} & \cell{3} & \cell{3} & \cell{3} & \cellbf{0.1037} & \cellbf{-} & \cellbf{0.1566} & \cell{9} & \cell{366.890} \\ \hline
\cell{NN15}      & \cell{3} & \cell{2} & \cell{2} & \cell{4} & \cell{1} & \cellbf{0.1039} & \cellbf{-} & \cellbf{0.1194} & \cell{6} & \cell{3.740} \\ \hline
\cell{NN16}      & \cell{8} & \cell{4} & \cell{4} & \cell{4} & \cell{8} & \cellbf{0.9557} & \cellbf{-} & \cellbf{0.9656} & \cell{48} & \cell{37.950} \\ \hline
\cell{NN17}      & \cell{3} & \cell{1} & \cell{2} & \cell{2} & \cell{1} & \cellbf{11.2182}& \cellbf{-} & \cellbf{11.2381}& \cell{117} & \cell{32.160} \\ \hline
\end{tabular}
\end{scriptsize}
\end{center}
\vskip -0.4cm
\end{table}
The numerical results computed by HIFOO and PENBMI are also included in Table \ref{tb:H_inf_problems}.
Here, the notation is the same as in Table \ref{tb:H2_prob}, except that $\mathcal{H}_{\infty}$ denotes the $\mathcal{H}_{\infty}$-norm of the closed-loop
system for the static output feedback
controller. We can see from Table \ref{tb:H_inf_problems} that the optimal values reported by Algorithm \ref{alg:A1} and HIFOO are almost similar for many
problems whereas in general PENBMI has
difficulties in
finding a feasible solution.

\subsection{$\mathcal{H}_2/\mathcal{H}_{\infty}$ optimization: BMI formulation}\label{subsec:H2Hinf_prob}
Motivated from the $\mathcal{H}_2$ and $\mathcal{H}_{\infty}$ optimization problems, in this subsection we consider the mixed
$\mathcal{H}_2/\mathcal{H}_{\infty}$ synthesis problem.
Let us assume that $D_{11} = 0$, $D_{21} = 0$ and the performance output $z$ is divided in two components, $z_1$ and $z_2$.
Then the linear system \eqref{eq:LTI} becomes:
\begin{equation}\label{eq:mixed_LTI_z1z2}
\left\{\begin{array}{cl}
&\dot{x} = Ax + B_1w + Bu,\\
&z_1 = C_1^{z_1}x + D_{12}^{z_1}u,\\
&z_2 = C_1^{z_2}x + D_{12}^{z_2}u,\\
&y   = Cx.
\end{array}\right.
\end{equation}
The mixed $\mathcal{H}_2/\mathcal{H}_{\infty}$ control problem is to find a static output feedback gain $F$ such that, for $u = Fy$, the $\mathcal{H}_2$-norm of
the closed loop from $w$ to $z_2$ is
minimized, while the $\mathcal{H}_{\infty}$-norm from $w$ to $z_1$ is less than some imposed level $\gamma$ \cite{Boyd1994,Leibfritz2004,Ostertag2008}.

This problem leads to the following optimization problem with BMI constraints \cite{Ostertag2008}:
\begin{equation}\label{eq:H2Hinf_prob}
\begin{array}{cl}
\displaystyle\min_{F, P_1,P_2,Z} \!\!\!\!\!\!\!\! & \textrm{trace}(Z)\\
\textrm{s.t.}\!\!\!\!\!\!\!\!  &\begin{bmatrix}A_F^TP_1 + P_1A_F + (C_{F}^{z_1})^TC_{F}^{z_1} & P_1B_1\\ B_1^TP_1 & -\gamma^2I\end{bmatrix}\prec 0,\\
&\begin{bmatrix}A_F^TP_2 + P_2A_F & P_2B_1\\B_1^TP_2 & -I\end{bmatrix}\prec 0, ~~
\begin{bmatrix}P_2 & (C_{F}^{z_2})^T\\C_{F}^{z_2} & Z \end{bmatrix}\succ 0, ~ P_1 \succ 0, ~P_2\succ 0,
\end{array}
\end{equation}
where $A_F := A + BFC$, $C_{F}^{z_1} := C_1^{z_1} + D_{12}^{z_1}FC$ and $C_{F}^{z_2} := C_1^{z_2} + D_{12}^{z_2}FC$.
Note that if $C = I_{n_x}$, the identity matrix, then this problem becomes a mixed $\mathcal{H}_2/\mathcal{H}_{\infty}$ of static state feedback design problem
considered in \cite{Ostertag2008}.
In this subsection, we test Algorithm \ref{alg:A1} for the static state feedback and output feedback cases.

\vskip 0.2cm
\noindent\textbf{Case 1. }\textit{The static state feedback case ($C = I_{n_x}$)}.
First, we apply the method in \cite{Ostertag2008} to find an initial point via solving two optimization problems with LMI constraints.
Then, we use the same approach as in the previous subsections to transform problem \eqref{eq:H2Hinf_prob} into an optimization problem with psd-convex-concave
matrix inequality constraints.
Finally, Algorithm \ref{alg:A1} is implemented to solve the resulting problem. For convenience of implementation, we introduce a slack variable $\eta$ and
then replace the objective function in
\eqref{eq:BMI_Hinf_prob} by $f(x) = \eta^2$ with an additional constraint $\textrm{trace}(Z) \leq \eta^2$.

In the first case, we test Algorithm \ref{alg:A1} with three problems. The first problem was also considered in \cite{Hassibi1999} with
\begin{align*}
&A   = {\footnotesize\begin{bmatrix}-1.40 & -0.49 & -1.93 \\ -1.73 & -1.69 & -1.25 \\ 0.99 & 2.08 & -2.49\end{bmatrix}}, ~
B_1  = {\footnotesize\begin{bmatrix}-0.16 & -1.29 \\ 0.81 & 0.96 \\ 0.41 & 0.65 \end{bmatrix}}, ~ B = {\footnotesize\begin{bmatrix} 0.25 \\ 0.41 \\ 0.65
\end{bmatrix}},\\
&C_1^{z_1} = {\footnotesize\begin{bmatrix} -0.41 & 0.44 & 0.68 \end{bmatrix}}, ~ C_1^{z_2} = {\footnotesize\begin{bmatrix} -1.77 & 0.50 & -0.40
\end{bmatrix}}, ~ D_{12}^{z_1} = D_{12}^{z_2} = 1,
~\textrm{and}~ \gamma = 2.
\end{align*}
If the tolerance $\varepsilon = 10^{-3}$ is chosen then Algorithm \ref{alg:A1} converges after $17$ iterations and reports the value $\eta = 0.7489$ with
$F = {\footnotesize \begin{bmatrix}
1.9485  & 0.3990  & -0.2119
\end{bmatrix}.}$
This result is similar to the one shown in \cite{Ostertag2008}.
If we regularize the subproblem \eqref{eq:psd_convex_subprob} with $\rho = 0.5\times 10^{-3}$ and $Q=I_{PF}$ then the number of iterations is reduced to $10$
iterations.

The second problem is \texttt{DIS4} in COMP$\textrm{l}_{\textrm{e}}$ib \cite{Leibfritz2003}. In this problem, we set $C_1^{z_1} = C_1^{z_2}$ and $D_{12}^{z_1}
= D_{12}^{z_2}$ as in
\cite{Ostertag2008}. Algorithm \ref{alg:A1} converges after $24$ iterations with the same tolerance $\varepsilon = 10^{-3}$. It reports $\eta = 1.6925$ and
$\gamma = 1.1996$ with
\begin{equation*}
F = {\footnotesize\begin{bmatrix}
  -0.8663 \!&\!  -0.6504 \!&\!  -1.1115 \!&\!  -0.1951 \!&\!  -0.6099 \!&\!   0.2065 \\
   0.1591 \!&\!  -0.4941 \!&\!  -0.6322 \!&\!  -0.5409 \!&\!  -1.2895 \!&\!   0.2774 \\
  -0.7017 \!&\!  -0.0785 \!&\!   0.6121 \!&\!  -0.8919 \!&\!   0.2518 \!&\!  -0.2354 \\
  -0.0522 \!&\!  -0.5556 \!&\!  -0.5838 \!&\!   0.4497 \!&\!  -1.4279 \!&\!  -0.6677
\end{bmatrix}}.
\end{equation*}
If we regularize the subproblem \eqref{eq:psd_convex_subprob} with $\rho = 0.5\times 10^{-3}$ and $Q= I_{PF}$ then the number of iterations is  $18$
iterations.

The third problem is \texttt{AC16} in COMP$\textrm{l}_{\textrm{e}}$ib \cite{Leibfritz2003}. In this example we also choose $C_1^{z_1} = C_1^{z_2}$ and
$D_{12}^{z_1} = D_{12}^{z_2}$ as in the previous
problem.
As mentioned in \cite{Ostertag2008}, if we choose a starting value $\gamma_0 = 100$, then the LMI problem can not be solved by the SDP solvers (e.g., Sedumi,
SDPT3) due to numerical problems.
Thus, we rescale the LMI constraints using the same trick as in \cite{Ostertag2008}.
After doing this, Algorithm \ref{alg:A1} converges after $298$ iterations with the same tolerance $\varepsilon = 10^{-3}$. The value of $\eta$  reported in
this case is $\eta = 12.3131$ and
$\gamma = 20.1433$  with
\begin{equation*}
F = {\footnotesize\begin{bmatrix}
   -1.8533 &  0.1737 &  0.6980 &  6.4208 \\
    4.2672 & -0.9668 & -1.5952 & -2.9240
\end{bmatrix}}.
\end{equation*}
The results obtained by Algorithm \ref{alg:A1} for solving problems \texttt{DIS4} and \texttt{AC16} in this paper confirm the results reported in
\cite{Ostertag2008}.

\noindent\textbf{Case 2. }\textit{The static output feedback case.}
Similarly to the previous subsections, we first propose a technique to determine a starting point for Algorithm \ref{alg:A1}.
We described this phase algorithmically as follows.

\noindent\textbf{Phase 1. }(\textit{Determine a starting point $x^0$}).

\noindent\textit{Step 1.} If $\alpha_0(A^T+A) < 0$ then set $F^0 = 0$. Otherwise, go to Step 3.

\noindent\textit{Step 2.} Solve the following linear SDP problem:
\begin{equation}\label{eq:H2Hinf_LMI_prob}
\begin{array}{cl}
\displaystyle\min_{P_1,P_2, Z} \!\!\!\!\!\!\!\! & \textrm{trace}(Z)\\
\textrm{s.t.}\!\!\!\!\!\!\!\!  &\begin{bmatrix}A_{F^0}^TP_1 + P_1A_{F^0} + (C_{F^0}^{z_1})^TC_{F^0}^{z_1} & P_1B_1\\ B_1^TP_1 & -\gamma^2I\end{bmatrix}\prec
0,\\
&\begin{bmatrix}A_{F^0}^TP_2 + P_2A_{F^0} & P_2B_1\\B_1^TP_2 & -I\end{bmatrix}\prec 0, ~\begin{bmatrix}P_2 & (C_{F^0}^{z_2})^T\\C_{F^0}^{z_2} & Z
\end{bmatrix}\succ 0, ~~ P_1 \succ 0, ~P_2\succ 0,
\end{array}
\end{equation}
where $A_{F^0} = A + BF^0C$, $C_{F^0}^{z_1} = C_1^{z_1}+D_{12}^{z_1}F^0C$ and $C_{F^0}^{z_2} = C_1^{z_2}+D_{12}^{z_2}F^0C$. If this problem has an optimal
solution $P_1^0, P_2^0$ and
$Z^0$ then terminate Phase 1. Set $x^0 := (F^0, P_1^0, P_2^0, Z^0)$ for a starting point of Algorithm \ref{alg:A1}. Otherwise go to Step 3.

\noindent\textit{Step 3. } Solve the following LMI feasibility problem:
\begin{equation*}\label{eq:H2Hinf_LMI_phase1}
\begin{array}{cl}
&\textrm{Find $Q\succ 0$, $W$ and $Z$ such that:}\\
&~~\begin{bmatrix}
AQ \!+\! QA^T\!\! +\! BW \!+\! W^T\!B^T \!\!&\!\! B_1 \!\!&\!\! (C_1 \!+\! D_{12}W)^T\!\! \\\!\! B_1^T \!\!&\!\! -I_w \!\!&\!\! O \!\!\\ \!\!C_1 \!+\! D_{12}W
\!\!&\!\! O &\!\! -\gamma^2I_z\!\!
\end{bmatrix} \!\prec\! 0,\\
&~~\begin{bmatrix}AQ + QA^T + BW + W^TB^T & B_1 \\ B_1^T & -I_w\end{bmatrix}\prec 0,~
\begin{bmatrix}Q & (C_1Q + D_{12}W)^T \\C_1Q + D_{12}W & Z \end{bmatrix}\succ 0,\\
\end{array}
\end{equation*}
to obtain a solution $Q^{*}$, $W^{*}$ and $Z^{*}$. Set $F^{*} := W^{*}(Q^{*})^{-1}C^{+}$, where $C^{+}$ is the pseudo-inverse of $C$. Solve again problem
\eqref{eq:H2Hinf_LMI_prob} with
$F^0 := F^{*}$. If problem \eqref{eq:H2Hinf_LMI_prob} has solution then terminate Phase 1. Otherwise, perform Step 4.

\noindent\textit{Step 4. } Solve the following optimization with BMI constraints:
\begin{equation}\label{eq:H2Hinf_BMI_prob}
\begin{array}{cl}
\displaystyle\max_{\beta, F, P_1,P_2} \!\!\!\!\!\!\!\! &\!\!\! \beta\\
\!\!\!\!\!\!\!\!\textrm{s.t.}\!\!\!\!\!\!\!\! & \!\!\!P_1 \succ 0, ~P_2\succ 0,\\
\!\!\!\!\!\!\!\! & \!\! A_F^T\!P_1 \!+\! P_1A_F^T \! +\! (C_{F}^{z_1})^T\!C_{F}^{z_1} + \frac{1}{\gamma^2}P_1B_1B_1^TP_1\preceq -2\beta P_1,\\
\!\!\!\!\!\!\!\! & \!\! A_F^TP_2 + P_2A_F + P_2B_1B_1^TP_2 \preceq -2\beta P_2
\end{array}
\end{equation}
to obtain an optimal solution $F^{*}$ corresponding to the optimal value $\beta^{*}$.
If $\beta^{*} > 0$ then set $F^0 := F^{*}$ and go back to Step 2 to determine $P_1^0$, $P_2^0$ and $Z^0$. Otherwise, declare that no strictly feasible point of
problem \eqref{eq:H2Hinf_prob} is found.
\eofproof

Since at Step 4 of Phase 1, it requires to solve an optimization problem with two BMI constrains. This task is usually expensive. In our implementation, we only
terminate this step after find a
strictly feasible point with a feasible gap $0.1$.
If matrix $C$ is invertible then the matrix $F^{*}$ at Step 3 is $F^{*} = W^{*}(Q^{*})^{-1}C^{-1}$. Hence, we can ignore Step 4 of Phase 1.

To avoid the numerical problem in Step 3, we can reformulate problem \eqref{eq:H2Hinf_LMI_phase1} equivalently to the following one:
\begin{equation*}\label{eq:H2Hinf_LMI2_phase1}
\begin{array}{cl}
&\textrm{Find $Q\succ 0$, $W$ and $Z$ such that:}\\
&~\begin{bmatrix}
AQ \!+\! QA^T\!\! +\! BW \!+\! W^T\!B^T \!\!\!\!&\!\!\!\! B_1 \!\!&\!\! (C_1 \!+\! D_{12}W)^T\!\! \\\!\! B_1^T \!\!\!\!&\!\!\!\! -\gamma I_w \!\!&\!\! O \!\!\\
\!\!C_1 \!+\! D_{12}W \!\!&\!\! O &\!\!
-\gamma I_z\!\!
\end{bmatrix} \!\prec\! 0,~
\begin{bmatrix}Q \!\!\! &\!\!\! (C_1Q + D_{12}W)^T \\C_1Q + D_{12}W \!\!\!&\!\!\! Z \end{bmatrix} \! \succ\! 0.\\
\end{array}
\end{equation*}
We test the algorithm described above for several problems in COMP$\textrm{l}_{\textrm{e}}$ib with the level values $\gamma = 4$ and $\gamma = 10$. In these
examples, we assume that the output signals
$z_1\equiv z_2$. Thus we have $C_1^{z_1} = C_1^{z_2} = C_1$ and $D_{12}^{z_1} = D_{12}^{z_2} = D_{12}$.
The parameters and the stopping criterion of the algorithm are chosen as in Subsection \ref{subsec:H2_prob}. The computational results are reported in Table
\ref{tb:H2Hinf_problems} with $\gamma = 4$
and $\gamma = 10$.
\begin{table}[!ht]
\begin{center}
\renewcommand{\arraystretch}{0.7}
\begin{scriptsize}
\caption{$\mathcal{H}_2/\mathcal{H}_{\infty}$ synthesis benchmarks on COMP$\textrm{l}_{\textrm{e}}$ib plants}\label{tb:H2Hinf_problems}
\newcommand{\cell}[1]{{\!}#1{\!}}
\newcommand{\cellbf}[1]{{\!}\textbf{#1}{\!}}
\begin{tabular}{|l|r|r|r|r|r|r|r|}\hline
\multicolumn{1}{|c|}{\!\!\!\!\! Problem\!\!\!\!} & \multicolumn{3}{c|}{\!\!\!\!\!{Results and Performances ($\gamma=4$)}\!\!\!\!\!} &
\multicolumn{3}{|c|}{\!\!\!\!\!{Results and
Performances ($\gamma=10$)}\!\!\!\!\!} \\ \hline
\texttt{Name}\!\! &\!\!\!$\mathcal{H}_2/\mathcal{H}_{\infty}$~~~~~~\!\!\!& \!\!\!\textrm{iter}\!\!\! & \!\!\!\texttt{time\![s]}\! &
\!$\mathcal{H}_2/\mathcal{H}_{\infty}$~~~~~\!&
\!\!\! \texttt{iter} \!\!\! & \!\!\texttt{time\![s]}\!\!\! \\ \hline
\cell{{~}AC1} & \cellbf{0.0585~/~0.0990}  & \cell{3}  & \cell{4.22}   & \cellbf{0.0585~/~0.0990} & \cell{3}  & \cell{4.27} \\ \hline 
\cell{{~}AC2} & \cellbf{0.1067~/~0.1723}  & \cell{6}  & \cell{7.31}   & \cellbf{0.1070~/~0.1727} & \cell{3}  & \cell{7.15} \\ \hline 
\cell{{~}AC3} &\cellbf{5.2770~/~3.9999}   & \cell{51} & \cell{281.53} & \cellbf{4.5713~/~5.1298} & \cell{18} & \cell{19.18} \\ \hline
\cell{{~}AC6} & \cellbf{-~/~-}            & \cell{- } & \cell{-}      & \cellbf{4.0297~/~4.8753} & \cell{283}& \cell{330.64} \\ \hline
\cell{{~}AC7} & \cellbf{0.0415~/~0.0961}  & \cell{1}  & \cell{3.39}   & \cellbf{0.0420~/~0.1286} & \cell{2}  & \cell{3.91} \\ \hline
\cell{{~}AC8} & \cellbf{1.2784~/~2.2288}  & \cell{43} & \cell{60.78}  & \cellbf{1.3020~/~2.5719} & \cell{23} & \cell{31.59} \\ \hline
\cell{{~}AC11}& \cellbf{4.0704~/~4.0000}  & \cell{76} & \cell{175.75} & \cellbf{4.0021~/~5.1949} & \cell{117}& \cell{122.86} \\ \hline
\cell{{~}AC12}& \cellbf{0.0924~/~0.3486}  & \cell{18} & \cell{73.46}  & \cellbf{1.4454~/~1.6444} & \cell{300}& \cell{234.13} \\ \hline
\cell{{~}AC17}& \cellbf{-~/~ - }          & \cell{-}  & \cell{-}      & \cellbf{4.1228~/~6.6472} & \cell{2} & \cell{11.620} \\ \hline
\cell{{~}HE1} & \cellbf{0.1123~/~0.2257}  & \cell{2}  & \cell{131.18} & \cellbf{0.0973~/~0.2080} & \cell{1}  & \cell{30.97} \\ \hline
\cell{{~}HE2} & \cellbf{- ~/~ -}          & \cell{-}  & \cell{-}      & \cellbf{4.7302~/~9.8931} & \cell{75} & \cell{55.48} \\ \hline
\cell{{~}REA1}& \cellbf{1.8214~/~1.4740}  & \cell{30} & \cell{25.64}  & \cellbf{1.8213~/~1.4730} & \cell{30} & \cell{26.65} \\ \hline
\cell{{~}REA2}& \cellbf{3.5014~/~3.5180}  & \cell{42} & \cell{22.09}  & \cellbf{3.5015~/~3.5209} & \cell{45} & \cell{23.26} \\ \hline
\cell{{~}DIS1}& \cellbf{- ~/~ -}          & \cell{-}  & \cell{-}      & \cellbf{2.8505~/~4.7904} & \cell{15} & \cell{30.51} \\ \hline
\cell{{~}DIS2}& \cellbf{1.5079~/~1.8500}  & \cell{18} & \cell{7.92}   & \cellbf{1.5079~/~1.8520} & \cell{21} & \cell{7.92} \\ \hline
\cell{{~}DIS3}& \cellbf{2.0577~/~1.7934}  & \cell{27} & \cell{25.03}  & \cellbf{2.0577~/~1.7766} & \cell{30} & \cell{24.54} \\ \hline
\cell{{~}DIS4}& \cellbf{1.6926~/~1.1952}  & \cell{21} & \cell{18.62}  & \cellbf{1.6926~/~1.2009} & \cell{24} & \cell{21.55} \\ \hline
\cell{{~}AGS} & \cellbf{- ~ / ~ -}        & \cell{-}  & \cell{-}      & \cellbf{7.0332~/~8.2035} & \cell{8}  & \cell{196.73} \\ \hline
\cell{{~}PSM} & \cellbf{1.5115~/~0.9248}  & \cell{177}& \cell{160.41} & \cellbf{1.5115~/~0.9248} & \cell{180}& \cell{167.31} \\ \hline
\cell{{~}EB2} & \cellbf{0.7765~/~1.0828}  & \cell{7}  & \cell{9.70}   & \cellbf{0.7768~/~1.0867} & \cell{10} & \cell{13.16} \\ \hline
\cell{{~}EB3} & \cellbf{0.8406~/~0.9249}  & \cell{1}  & \cell{3.21}   & \cellbf{0.8383~/~0.9418} & \cell{1}  & \cell{2.93} \\ \hline
\cell{{~}EB4} & \cellbf{1.0147~/~1.0707}  & \cell{6}  & \cell{59.55}  & \cellbf{0.9981~/~1.2146} & \cell{12} & \cell{111.26} \\ \hline
\cell{{~}NN2} & \cellbf{1.5651~/~2.4834}  & \cell{12} & \cell{5.37}   & \cellbf{1.5651~/~2.4876} & \cell{12} & \cell{5.49} \\ \hline
\cell{{~}NN4} & \cellbf{1.8778~/~2.0501}  & \cell{202}& \cell{154.49} & \cellbf{1.8779~/~2.0519} & \cell{213}& \cell{161.00} \\ \hline
\cell{{~}NN8} & \cellbf{2.3609~/~3.9999}  & \cell{21} & \cell{15.71}  & \cellbf{2.3376~/~4.6514} & \cell{15} & \cell{6.57} \\ \hline
\cell{{~}NN15}& \cellbf{0.0820~/~0.1010}  & \cell{42} & \cell{18.75}  & \cellbf{0.0771/0.1012}   & \cell{24} & \cell{10.47} \\ \hline
\cell{{~}NN16}& \cellbf{0.3187~/~0.9574}  & \cell{90} & \cell{96.44}  & \cellbf{0.3319~/~0.9572} & \cell{258}& \cell{303.87} \\ \hline
\end{tabular}
\end{scriptsize}
\end{center}
\vskip -0.4cm
\end{table}
Here, $\mathcal{H}_2/\mathcal{H}_{\infty}$ are the $\mathcal{H}_2$ and $\mathcal{H}_{\infty}$ norms of the closed-loop systems for the static output feedback
controller, respectively.
With $\gamma = 10$, the computational results show that Algorithm \ref{alg:A1} satisfies the condition $\norm{P_{\infty}(s)}_{\infty} \leq \gamma = 10$ for all
the test problems.
While, with $\gamma = 4$, there are $5$ problems reported infeasible, which are denoted by ``-''. The $\mathcal{H}_{\infty}$-constraint of three problems: AC3,
AC11 and NN8 is active with respect to
$\gamma = 4$.

\section{Concluding remarks}
We have proposed a new algorithm for solving many classes of optimization problems involving BMI constraints arising in static feedback controller design. The
convergence of the algorithm has been
proved under standard assumptions. Then, we have applied our method to design static feedback controllers for various problems in robust control design.
The algorithm is easy to implement using the current SDP software tools. The numerical results are also reported for the benchmark collection in
COMP$\textrm{l}_{\textrm{e}}$ib.
Note, however, that our method depends crucially on the psd-convex-concave decomposition of the BMI constraints. In practice, it is important to look at the
specific structure of the problems and find
an appropriate psd-convex-concave decomposition for Algorithm \ref{alg:A1}. The method proposed can be extended to general nonlinear semidefinite programming,
where the psd-convex-concave
decomposition of the nonconvex mappings are available. From a control design point of view, the application to more general reduced order controller synthesis
problems  and the extension towards
linear parameter varying  or time-varying systems are future research directions.

\section{Acknowledgments}
{\footnotesize
Research supported by Research Council KUL: CoE EF/05/006 Optimization in Engineering(OPTEC), IOF-SCORES4CHEM, GOA/10/009 (MaNet), GOA/10/11, several
PhD/postdoc and fellow grants; Flemish
Government:
FWO: PhD / postdoc grants, projects G.0452.04, G.0499.04, G.0211.05, G.0226.06,
G.0321.06, G.0302.07, G.0320.08, G.0558.08, G.0557.08, G.0588.09, G.0377.09, G.0712.11, research communities (ICCoS, ANMMM, MLDM); IWT: PhD Grants, Belgian
Federal Science Policy Office: IUAP P6/04; EU: ERNSI; FP7-HDMPC, FP7-EMBOCON, Contract Research: AMINAL. Other: Helmholtz-viCERP, COMET-ACCM.
}

\appendix
\subsection*{Proof of Lemma \ref{le:descent_dir}.}\label{app:A1}
For any matrices $A, B\in\mathcal{S}^p_{+}$, we have $A\circ B \geq 0$.
From \textit{Step 1} of Algorithm \ref{alg:A1}, we have $x^{k+1}$ is the solution of the convex subproblem \eqref{eq:psd_convex_subprob} and $\Lambda^{k+1}$ is
the corresponding multiplier, under
Assumption \ref{as:A3}, they must satisfy the following generalized Kuhn-Tucker condition:
\begin{equation}\label{eq:KT_cond}
\left\{\begin{array}{cl}
&\!\! 0\in \nabla f(x^{k\!+\!1}) \!+\! \rho_kQ^T_kQ_k(x^{k\!+\!1} \!-\! x^k) \!+\! \left\{\sum_{i=1}^lD[G_i(x) \!-\! H_i(x^k)\right.\\
&\!\! \left.- DH_i(x^k)(x-x^k)]|_{x^{k+1}}\right\}^{*}\Lambda^{k+1}_i + N_{\Omega}(x^{k+1}),\\
&\!\!G_i(x^{k+1}) - H_i(x^k) - DH(x^k)(x^{k+1}-x^k) \preceq 0, ~~ \Lambda_i \succeq 0,\\
&\!\! \left[\!G_i(x^{k\!+\!1}) \!-\! H_i(x^k) \!-\! DH(x^k)(x^{k\!+\!1}\!-\!x^k)\!\right]\circ \Lambda_i^{k\!+\!1} = 0.
\end{array}\right.
\end{equation}
Noting that $D\left[G_i(x) - H_i(x^k) - DH_i(x^k)(y-x^{k})\right]|_{x=x^{k+1}} = DG_i(x^{k+1}) - DH_i(x^k)$ for $i=1,\dots l$, it follows from the first line of
\eqref{eq:KT_cond} and the convexity of
$f$ that
\begin{align}\label{eq:term1}
&f(y) \!-\! f(x^{k\!+\!1}) \!+\! \left\{\sum_{i=1}^l[DG_i(x^{k\!+\!1})\!-\!DH_i(x^k)]^{*}\Lambda_i^{k+1}\right\}^T\!\!\!(y-x^{k+1})\notag\\
&\geq \left\{\nabla f(x^{k+1}) + \sum_{i=1}^l[DG_i(x^{k+1})-DH_i(x^k)]^{*}\Lambda_i^{k+1}\right\}^T\!\!\!(y-x^{k+1})\notag\\
&+\! \frac{\rho_f}{2}\norm{y \!-\! x^{k\!+\!1}}_2^2 \! \geq\! \frac{\rho_f}{2}\norm{y \!-\! x^{k\!+\!1}}^2_2 \!+\!
\rho_k(y\!-\!x^{k\!+\!1})^TQ_k^TQ_k(x^k\!-\!x^{k\!+\!1}),  ~~\forall y\in\Omega.
\end{align}
On the other hand, we have
\begin{align}\label{eq:term2}
&\left\{ [DG_i(x^{k + 1})  - DH_i(x^k)]^{*}\Lambda_i^{k + 1}\right\}^T\!(y - x^{k + 1}) \nonumber\\
& =  \Lambda_i^{k + 1} \circ
[DG_i(x^{k + 1})(y  - x^{k + 1}) - DH_i(x^k)(y - x^{k + 1})].
\end{align}
Since $G_i$ and $H_i$ are psd-convex, applying Lemma \ref{le:psd_convex_properties} we have
\begin{align*}
& G_i(x^k) \!-\! G_i(x^{k\!+\!1}) \!\succeq\! DG_i(x^{k\!+\!1}\!)(x^k \!-\! x^{k\!+\!1}), \nonumber\\
\textrm{and}~& H_i(x^{k\!+\!1}) \!-\! H_i(x^k) \!\succeq\! DH_i(x^k)(x^{k\!+\!1} \!-\! x^k),
~i=1,\dots, l.
\end{align*}
Summing up these inequalities we obtain
\begin{align*}
G_i(x^k) \!-\! H_i(x^k) \!-\! [G_i(x^{k\!+\!1}) \!-\! H_i(x^{k\!+\!1})]  \!\succeq\! [DG_i(x^{k \!+\! 1})(x^k \!-\! x^{k+1}) \!-\!
DH_i(x^k)(x^k-x^{k+1})].
\end{align*}
Using the fact that $\Lambda^{k+1}_i\succeq 0$, this inequality implies that
\begin{align}\label{eq:term3}
& \Lambda^{k + 1}_i  \circ  \left\{ G_i(x^k) - H_i(x^k)  - [G_i(x^{k + 1} ) -  H_i(x^{k\!+\!1})] \right\} \nonumber\\
& \succeq  \Lambda^{k\!+\!1}_i  \circ [DG_i(x^{k\!+\!1})(x^k - x^{k\!+\!1})  -  DH_i(x^k)(x^k  -  x^{k\!+\!1})].
\end{align}
Substituting $y=x^k$ into \eqref{eq:term1} and then combining the consequence, \eqref{eq:term2}, \eqref{eq:term3} and the last line of
\eqref{eq:KT_cond} to obtain
\begin{align}\label{eq:term4}
f(x^k) \!-\! f(x^{k\!+\!1}) \!+\! \sum_{i=1}^l\Lambda_i^{k\!+\!1}\circ[G_i(x^k)-H_i(x^k)] \geq \frac{\rho_f}{2}\norm{x^{k \!+ \! 1}-x^k}_2^2 \!+\!
\rho_k\norm{Q_k(x^{k\!+\! 1} \!-\! x^k)}_2^2.
\end{align}
Now, since $x^k$ is the solution of the convex subproblem \eqref{eq:psd_convex_subprob} linearized at $x^{k-1}$. One has
$G_i(x^k) - H_i(x^k) \preceq 0$.
Moreover, since $\Lambda^{k+1}_i\succeq 0$, we have $\Lambda^{k+1}_i\circ\left[G_i(x^k) - H_i(x^k)\right] \leq 0$.
Substituting this inequality into \eqref{eq:term4}, we obtain
\begin{equation*}
f(x^k) - f(x^{k+1}) \geq \frac{\rho_f}{2}\norm{x^k-x^{k+1}}_2^2 + \rho_k\norm{Q_k(x^{k+1}-x^k)}_2^2.
\end{equation*}
This inequality is indeed \eqref{eq:descent_dir} which proves the item a).
If there exists at least one $i_0\in \left\{1,\dots, l\right\}$ such that $G_{i_0}(x^k) - H_{i_0}(x^k) \prec 0$ and $\Lambda^{k+1}_{i_0} \succ 0$ then
$\Lambda_{i_0}^{k+1}\circ\left[G_{i_0}(x^k) -
H_{i_0}(x^k)\right] < 0$. Substituting this inequality into \eqref{eq:term4} we conclude that $f(x^{k+1}) < f(x^k)$ which proves item b).
The last statement c) follows directly from the inequality \eqref{eq:descent_dir}.
\eofproof



\end{document}